\documentclass[1p]{elsarticle}

\usepackage{lineno}
\usepackage[colorlinks,linkcolor={blue}]{hyperref}
\modulolinenumbers[5]

\usepackage{amsfonts}
\usepackage{graphicx}
\usepackage{amsmath}
\usepackage{amssymb}
\usepackage[usenames, dvipsnames]{color}

\usepackage{mathtools}

\newtheorem{theorem}{Theorem}

\newtheorem{corollary}[theorem]{Corollary}

\newtheorem{definition}[theorem]{Definition}

\newtheorem{lemma}[theorem]{Lemma}
\newtheorem{notation}[theorem]{Notation}

\newtheorem{proposition}[theorem]{Proposition}

\newtheorem{remark}[theorem]{Remark}
\newenvironment{proof}[1][Proof]{\textbf{#1.} }{\ \rule{0.5em}{0.5em}}

\journal{...}

\begin{document}
\begin{frontmatter}
\title{A non-perturbative topological persistence theorem for normally hyperbolic invariant manifolds}

\author{Wojciech Spólnik\footnote{The author to whom any correspondence should be addressed. \newline
\textit{E-mail address:} \texttt{wspolnik@agh.edu.pl}}, Maciej J. Capiński}

\address{Faculty of Applied Mathematics, AGH University of Krakow}

\begin{abstract}
Floer proved that, under a sufficiently small $C^0$ perturbation, a
normally hyperbolic invariant manifold gives rise to a nearby invariant set whose cohomology ring contains a subring isomorphic to that of the original manifold. Using topological conditions, expressed
in terms of the so-called covering relations, we generalize this theorem to a non-perturbative context. Our result does not require the existence of an invariant manifold before perturbation, nor does it require the original map to be smooth or invertible.
For the proof, we develop an additional structure on the Conley index for maps on vector bundles.
\end{abstract}

\begin{keyword} normally hyperbolic invariant manifolds, cohomology of invariant sets, covering relations, Conley index theory
\MSC[2010] 
37D05,  	
37D10, 	
37B30. 	
\end{keyword}
\end{frontmatter}


\section{Introduction}

The paper presents a topological construction that generalises the classical result for normally hyperbolic invariant manifolds (NHIMs) to the non-perturbative setting. Before outlining our results, we provide some context.

Let $M$ be a smooth manifold. A manifold $\Lambda\subset M,$ which is
invariant under a $C^{1}$ diffeomorphism $f:M\rightarrow M$ is said to be \emph{normally hyperbolic} if the
tangent bundle at $M$ restricted to $\Lambda$ admits a $df$-invariant splitting%
\begin{equation}
T_{\Lambda}M=T\Lambda\oplus E^{u}\oplus E^{s}\label{eq:splitting}%
\end{equation}
and there exist constants $C>0$ and $0\,<\lambda<\mu^{-1}<1$ (known as the
\emph{rates}) such that%
\begin{align}
\left\vert \left\vert df_{\theta}^{-k}\left(  x\right)  \right\vert
\right\vert  &  \leq C\lambda^{k}\left\vert \left\vert x\right\vert
\right\vert \text{ for }k\geq0\text{, }\theta\in\Lambda\text{, }x\in
E_{\theta}^{u}\nonumber\\
\left\vert \left\vert df_{\theta}^{k}\left(  y\right)  \right\vert
\right\vert  &  \leq C\lambda^{k}\left\vert \left\vert y\right\vert
\right\vert \text{ for }k\geq0\text{, }\theta\in\Lambda\text{, }y\in
E_{\theta}^{s},\label{eq:rate-cond}\\
\left\vert \left\vert df_{\theta}^{k}\left(  z\right)  \right\vert
\right\vert  &  \leq C\mu^{\left\vert k\right\vert }\left\vert \left\vert
z\right\vert \right\vert \text{ for }k\in\mathbb{Z}\text{, }\theta\in
\Lambda\text{, }z\in T_{\theta}\Lambda.\nonumber
\end{align}

The classical results on normally hyperbolic theory state that a NHIM is
robust under sufficiently small $C^{1}$ perturbation of the map \cite{HPS1977}. However, the splitting \eqref{eq:splitting} into contracting and expanding directions alone is not sufficient to guarantee persistence. The key requirement are the rate conditions
\eqref{eq:rate-cond}. In fact, Jarn\'{\i}k and Kurzweil \cite{JarnikKurzweil1969} gave an example showing that in the absence of rate conditions, the invariant manifold can break down under a small $C^0$ perturbation to an invariant set which is not even a topological manifold. Moreover, Ma\~n\'e \cite{Mane}
showed that persistent manifolds have to be normally hyperbolic. 

In spite of that, one would like to find out whether some properties of the original invariant manifold  persist without assuming rate conditions. One notable result in this direction is the work of Floer \cite{FloerPerThm}, where he considers an invariant manifold with a splitting \eqref{eq:splitting} and shows that after a sufficiently small $C^{0}$
perturbation of the system there exists an invariant set (not necessarily a
manifold), whose cohomology ring contains that of the original manifold as a subring. The disadvantage of this result is its perturbative formulation.

The main contribution of this paper is a generalization of Floer's result to a purely topological, non-perturbative setting. Our assumptions are based on those considered in \cite{CapKub} and expressed in terms of the so-called covering relations (also referred
to in the literature as correctly aligned windows). We will consider a continuous (not necessarily invertible) map $f: E \rightarrow E$ on a smooth Banach vector bundle $E=E^{u}\oplus E^{s}$, whose base is a smooth, compact and connected
manifold without boundary $\Lambda.$ In particular, we do not have to assume that $\Lambda$ is invariant with respect to $f$ or that it preserves the splitting of the bundle.

In \cite{CapKub} it is shown that suitable covering relations
ensure the existence of an invariant set, which intersects every fibre of the bundle $E$. Moreover, it is proven there that these covering relations are satisfied
in a natural way in the context of $C^{0}$ perturbations of NHIMs. We will show that the cohomology ring of this invariant set contains a subring isomorphic to the cohomology ring of $\Lambda$. This result can be applied to systems in which the relevant invariant manifolds arise in settings that are far beyond the reach of analytical methods. The covering relations are readily verifiable by computing bounds on the images of sets, a task that can be performed efficiently using interval arithmetic and computer-assisted tools \cite{KapelaWilczakZgliczynski2022, KapMrWilZg2021}. In this way, we provide a tool that can establish the existence of an invariant set together with a lower bound for its cohomology.

Even though our proof shares similarities with the ideas developed by Floer, there are some significant differences. Both proofs are Conley-theoretic, but Floer's argument reduces the problem to the continuous-time case by considering the suspension flow of the original diffeomorphism and using the Conley index theory for flows \cite{FloerConInd, ConIsolated}. We believe that this approach offers limited flexibility in our case. To obtain our result, we choose to work in the phase space of the original system and use the Conley index theory for discrete-time dynamical systems \cite{Mrozek1990}. Our main technical tool is the external product on the Conley index defined in the setting of maps on vector bundles.

The practical motivation for our work is that a smooth dynamical system, modeled by a map or an ODE, is an approximation of reality. The true model is typically a non-smooth perturbation of this. From this perspective, our method is applicable to problems in which the existence of a NHIM plays a key role.

The paper is organized as follows. Section \ref{sec:main_result} contains a precise statement of the main results with some consequences and examples. Section \ref{sec:preliminaries} contains all preliminary facts that are necessary in the paper. In Section \ref{sec:ext_prod}, we develop the external product on the Conley index for maps on vector bundles and prove its properties. The proof of the main theorem is contained in Section \ref{sec:proof_cont_thm}. The paper ends with a short Section \ref{sec:conclusions} containing conclusions and acknowledgements.

\section{Statement of the main result}
\label{sec:main_result}

Let us assume that $\Lambda$ is a compact smooth $c$-dimensional manifold
without boundary, $E^{u},E^{s}$ are smooth Banach vector bundles\footnote{For
preliminaries on Banach vector bundles refer to Section
\ref{sec:banach_vc_bundles}.} over $\Lambda$ of ranks $u$, $s$, respectively, and
\[
E=E^{u}\oplus E^{s}.
\]

Throughout the work we will use the following notation. We will write
$\pi:E\rightarrow\Lambda$ for the projection of the Banach vector bundle. For $x\in E_{\theta}^{u}$ and
$y\in E_{\theta}^{s}$ and for $\theta\in\Lambda$ we will write%
\begin{equation}
\left(  \theta;x,y\right)  \in E_{\theta} \label{eq:theta-notation}%
\end{equation}
to denote the point in $E$. We decide to emphasize that
the point $\left(  x,y\right) $ is in the fiber $E_{\theta}$ by adding
$\theta$ to the notation (\ref{eq:theta-notation}). For $x\in E_{\theta}^{u}$
and $y\in E_{\theta}^{s}$, we will write $\Vert x\Vert_{u}$ and $\Vert
y\Vert_{s}$ for the norms in $E_{\theta}^{u}$ and $E_{\theta}^{s}$, respectively.

We define the following sets
\begin{align}
D  &  :=\{(\theta;x,y)\in E\mid\theta\in\Lambda,\ \Vert x\Vert_{u}%
\leq1,\ \Vert y\Vert_{s}\leq1\},\label{eq:set_D}\\
D^{-}  &  :=\{(\theta;x,y)\in E\mid\theta\in\Lambda,\ \Vert x\Vert
_{u}=1,\ \Vert y\Vert_{s}\leq1\},\nonumber\\
D^{+}  &  :=\{(\theta;x,y)\in E\mid\theta\in\Lambda,\ \Vert x\Vert_{u}%
\leq1,\ \Vert y\Vert_{s}=1\}.\nonumber
\end{align}
We will also use the notation
\[
B_{\theta}^{u}:=\{x\in E_{\theta}^{u}\mid\Vert x\Vert_{u}\leq1\}
\]
to denote a closed unit ball in a fiber $E_{\theta}^{u}.$

The following definition is a modification of covering relations introduced
by Zgliczy\'{n}ski and Gidea \cite{GidZglicz}, adapted to a vector bundle
setting in \cite{CapKub}. As we mentioned in the Introduction, it will constitute the main assumption needed for our result.

\begin{definition}
\label{def:cov_relations} Assume that $f:E\rightarrow E$ is a continuous
map. 
We say that $D$ $f$-covers itself, denoted $D\overset{f}
{\Longrightarrow}D$, if the following conditions are satisfied:

\begin{itemize}
\item[1)] there exists a homotopy
\begin{equation}
h:[0,1]\times D\rightarrow E \label{eq:cov_rel_homotopy}%
\end{equation}
such that
\begin{align*}
h(0,\cdot)  &  =f(\cdot),\\
h([0,1]\times D^{-})\cap D  &  =\emptyset,\\
h([0,1]\times D)\cap D^{+}  &  =\emptyset;
\end{align*}

\item[2)] there exists a continuous map $\eta:\Lambda\rightarrow\Lambda$ such
that
\begin{equation}
\deg\eta\neq0 \label{eq:cov_rel_degree}%
\end{equation}
if $E^u$ is an orientable vector bundle and $\Lambda$ is an orientable manifold, or 
\begin{equation}
\deg_2\eta\neq0 \label{eq:cov_rel_degree_}%
\end{equation}
otherwise; moreover
\begin{itemize}
\item[a)] if $u>0$, then for any $\theta\in\Lambda$ there exists a linear map
$A_{\theta}:E_{\theta}^{u}\rightarrow E_{\eta(\theta)}^{u}$ such that
\begin{equation}
A_{\theta}(\partial B_{\theta}^{u})\subset E_{\eta(\theta)}^{u}\setminus
B_{\eta(\theta)}^{u}\quad(A_{\theta}\text{ is expanding})
\label{eq:cov_lin_isom}%
\end{equation}
and
\[
h(1,(\theta;x,y))=(\eta(\theta);A_{\theta}x,0)\in E_{\eta(\theta)};
\]

\item[b)] if $u=0$, then
\[
h(1,(\theta;y))=(\eta(\theta);0)\in E_{\eta(\theta)}=E_{\eta(\theta
)}^{s}.
\]

\end{itemize}
\end{itemize}
\end{definition}

Intuitively, in Definition \ref{def:cov_relations} we require that the map $f$
has a topologically expanding direction $x$, a topologically contracting
direction $y$ and the dynamics on $\Lambda$ is homotopic to the map with
a non-zero degree. Figure \ref{fig:covering} contains the picture illustrating the geometric meaning of a covering. For examples of both analytical and computer-assisted verification of covering relations, see \cite{CapKub}.

\begin{figure}
    \centering
    \includegraphics[width=0.8\linewidth]{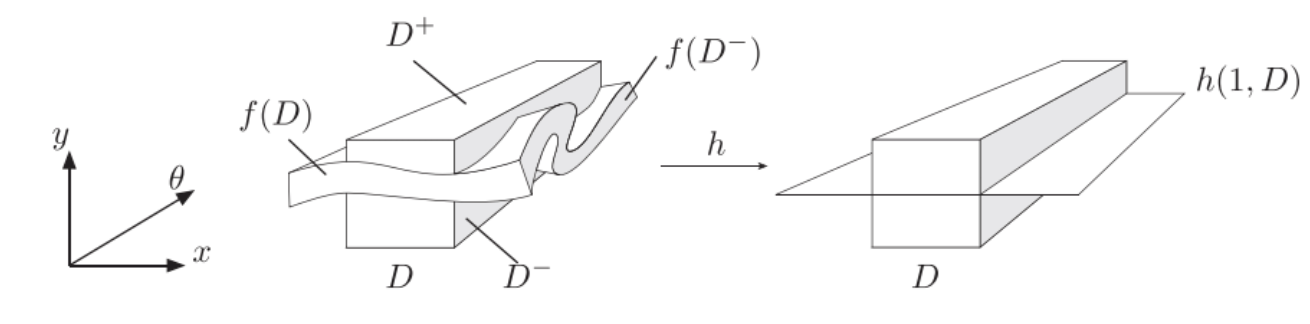}
    \caption{The geometric intuition behind covering relations. The picture is taken from \cite{CapKub}.}
    \label{fig:covering}
\end{figure}

Let us denote
\begin{equation*}
\operatorname{Inv}(D, f)=\left\{  p\in D\; \middle| \begin{array}{l}
\exists\left\{  q_{i}\right\}_{i\in
\mathbb{Z}} \subset D\text{ such that }q_{0}=p \\
\text{ and }f\left(q_{i}\right)
=q_{i+1}\text{ for all }i\in\mathbb{Z}
\end{array}\right\},
\end{equation*}
i.e. $\operatorname{Inv}(D, f)$ is a maximal invariant set of $f$ contained in $D$.

We are now ready to formulate the main result of this paper. 
\begin{theorem}
\label{thm:cont_thm}  Assume that $E^u$ is an orientable vector bundle and $\Lambda$ is an orientable manifold. Let $f:E\rightarrow E$ is a continuous map and
$D\overset{f}{\Longrightarrow}D$. Denote by $H^\ast(-; \mathbb{Q})$ the Alexander--Spanier cohomology functor with rational coefficients. Then the induced homomorphism
\begin{equation}
\tilde\pi^{\ast}:H^{\ast}(\Lambda; \mathbb{Q})\rightarrow H^{\ast}%
(\operatorname{Inv}(D, f); \mathbb{Q}) \label{eq:cont_thm_hom},
\end{equation}
where $\tilde\pi$ is a restriction of $\pi$ to $\operatorname{Inv}(D,f)$,
is injective.
In other words, the rational cohomology ring of $\operatorname{Inv}(D, f)$ has a subring isomorphic to the rational cohomology ring of $\Lambda$.
\end{theorem}
The proof of Theorem \ref{thm:cont_thm} will be given in Section
\ref{sec:proof_cont_thm}. Before we discuss this result, we would like to address the choice of the coefficient ring for cohomology.
\begin{remark}
\label{rem:cont_thm_coeff}
    The coefficient ring for cohomology in Theorem \ref{thm:cont_thm} can be chosen to be any field of non-zero characteristic.
    If, in addition, $f$ is a homeomorphism, one can choose coefficients arbitrarily.
\end{remark}
\begin{remark}
\label{rem:cont_thm_coeff_2}
    Even though Theorem \ref{thm:cont_thm} is formulated for orientable $E^u$ and $\Lambda$, the same conclusion hold for cohomology with $\mathbb{Z}_2$ coefficients in the non-orientable case.
\end{remark}
    
    We would like to point out that the same result holds in a more general situation. We consider the set $D$ of the form \eqref{eq:set_D} for simplicity. Our result holds in any setting that is homeomorphic to the above one. In applications, one can consider $D$ to be some neighborhood of a zero section of an invariant manifold in its normal bundle. In particular, a result from \cite{CapKub} shows that normal hyperbolicity implies the covering relations in some neighborhood of the NHIM, so Theorem \ref{thm:cont_thm} can be applied in the case of $C^0$ (explicit, potentially large) perturbations of NHIMs.

    The injection \eqref{eq:cont_thm_hom} can be used to obtain lower bounds for some topological invariants, like the cup-length or the Lyusternik--Schnirelmann category of the invariant set (see Cornea et al. \cite{Cornea}). A particularly interesting consequence is a lower bound for the Lebesgue covering dimension\footnote{For a topological space $X$, a \emph{Lebesgue covering dimension} is the smallest number $k$ such that every open cover $\mathcal{U}$ of $X$ admits a refinement $\mathcal{V}$ of order $k+1$, i.e.  $V_1 \cap \ldots \cap V_{k+1} = \emptyset$ for any pairwise different open sets $V_1, \ldots, V_{k+1} \in \mathcal{V}$. This notion extends the usual dimension for smooth manifolds.}. Note that our result remains true if we use \v{C}ech cohomology, since it coincides with Alexander--Spanier cohomology on paracompact Hausdorff spaces (see Spanier \cite{SpanierAlgTop}, Ch. 6, Section 9). One can show (see, for example, Engelking \cite{Engelking1978}, Ch. 1, § 9) that for any compact metric space $X$ the condition $\check{H}^n(X) \neq 0$ implies that $\dim X \geq n,$ where $\dim$ denotes the Lebesgue covering dimension and $\check{H}^n$ is the $n$-th \v{C}ech cohomology group. Thus in view of Theorem \ref{thm:cont_thm}, we have an inequality $\dim \operatorname{Inv(D,f)} \geq \operatorname{dim \Lambda}$, even though one could expect non-smooth perturbations of invariant manifolds to have a dust-like structure.  

As a final remark, we emphasize that the use of Alexander--Spanier cohomology is necessary, as it captures the overall structure of the invariant set and ignores local pathologies. Indeed, Theorem \ref{thm:cont_thm} is not true if we replace Alexander--Spanier cohomology with singular cohomology. As a simple counterexample, consider 
$E=S^1\times\mathbb R$ and $D=S^1\times[-1,1]$; here we identify \(S^1\) with \(\mathbb R/\mathbb Z\). (In this example we will have a center coordinate $\theta$ on $S^1$, with a one dimensional topologically expanding coordinate $x$ in $\mathbb R$ and without the topologically contracting coordinate $y$.) Consider the Warsaw circle
\[
W=
\left\{\left(\theta,\frac12\sin\frac{1}{\theta}\right):
0<\theta<1\right\}
\;\cup\;
\left(\{0\}\times\left[-\frac12,\frac12\right]\right)
\]
wrapped around the cylinder, and the signed distance (see Figure \ref{fig:counterexample})
\[
\rho(z)=
\begin{cases}
\operatorname{dist}(z,W),&z\in U_+,\\
0,&z\in W,\\
-\operatorname{dist}(z,W),&z\in U_-.
\end{cases}
\]
We define the map $f:E\rightarrow E$ by the formula
\[
f(\theta,x)=\bigl(\theta,x+\rho(\theta,x)\bigr),
\]
and then $\operatorname{Inv}(D,f)=W$. It is a standard fact in algebraic topology that $H^1_{\text{sing}}(W;\mathbb{Q}) = 0,$ the trivial group (see Munkres \cite{munkres}, Corollary 73.5), hence \eqref{eq:cont_thm_hom} would be the zero map on singular cohomology.

\begin{figure}
    \centering
    \includegraphics[width=0.5\linewidth]{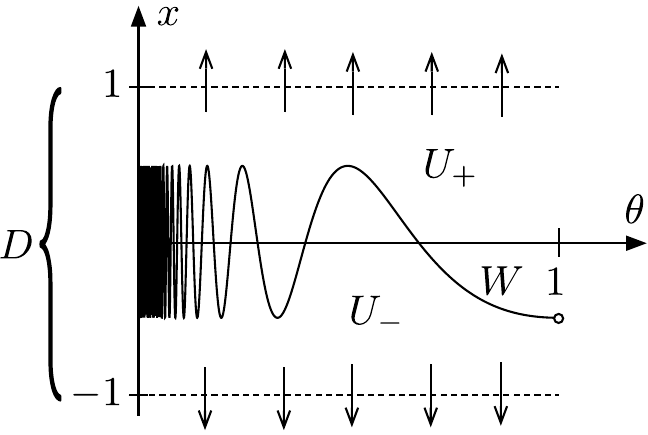}
    \caption{A counterexample showing that Theorem \ref{thm:cont_thm} is not true with singular cohomology.}
    \label{fig:counterexample}
\end{figure}


\section{Preliminaries}
\label{sec:preliminaries}

This section contains preliminary facts that will be used throughout the paper. Readers already familiar with some of the topics may safely skim or omit these
sections and refer to them as needed.

\subsection{Notations}
\label{subsec:notations}

We use the notation $\left(  X,A\right)  $ for a topological pair, i. e. $X$
is a topological space and $A\subset X$. We denote by $\partial_{X}A$,
$\operatorname{int}_{X}A$ and $\operatorname{cl}_{X}A$ the boundary, the
interior and the closure of $A$ in $X$, respectively. We omit the subscripts
if the ambient space is clear from the context. The space $X$ can be identified with the pair $(X, \emptyset).$ If $P = (P_1, P_2)$ and $Q=(Q_1,Q_2)$ are topological pairs, we write $P \subset Q$ if $P_i \subset Q_i$ for $i = 1, 2$.

Let $(X,A)$ and $(Y,B)$ be topological pairs. We write
\[
f:(X,A)\rightarrow(Y,B)
\]
if $f:X\rightarrow Y$ is a continuous map such that $f(A)\subset B$. If two maps $f$ and $g$ between topological spaces (or pairs) are homotopic, we will write $f \simeq g.$

Let us fix a commutative ring with identity $R$. By $H_{\ast}\left(
X,A;R\right)  $ we denote the singular homology graded $R$-module of the pair
$(X,A)$ with coefficients in $R.$ Similarly, by $H^{\ast}(X,A;R)$ we denote
the Alexander--Spanier cohomology graded $R$-module of the pair $(X,A)$ with
coefficients in $R.$ If $f:(X,A)\rightarrow(Y,B)$, then by
\begin{equation}
f^{\ast}:H^{\ast
}(Y,B;R)\rightarrow H^{\ast}(X,A;R) \label{eq:f*-cohomology-pair}%
\end{equation}
we mean a homomorphism induced by $f$ on cohomology. Similarly, $f_\ast$ means a homomorphism induced by $f$ on homology. 

If $f: P \rightarrow Q$ is a map of topological pairs (or spaces) and if $S \subset P, T \subset Q$ are topological pairs (or spaces) such that $f(S) \subset T$, then by $f_{S,T}$ we mean the map $f$ with the domain restricted to $S$ and the codomain restricted to $T$. If we only restrict the domain, we simply write $f_S$ instead of $f_{S,Q}$.

If $A\subset
B$, we will always use the notation
\begin{equation}
i_{A,B}:A\rightarrow B \label{eq:inclusion_map}%
\end{equation}
to denote the inclusion map. The same convention applies to the case when $A$
and $B$ are topological pairs.

Similarly,
\[
f_{\ast}:H_{\ast}(X,A;R)\rightarrow H_{\ast}(Y,B;R)
\]
denotes the homomorphism induced by $f$ on homology. We will write $f^{\ast
m}$ instead of $(f^{\ast})^{m}$ to ease the notation. When the coefficient
ring for (co)homology is clear from the context, we will use notations
$H^{\ast}(X,A)$ and $H_{\ast}\left(  X,A\right)  $.

If $\eta:\Lambda\rightarrow\Lambda$ is a (not necessarily smooth) map on a
smooth manifold $\Lambda$, then by $\deg_{2}\eta$ we denote its degree modulo
2. When $\Lambda$ is orientable, we denote by $\deg\eta$ its Brouwer degree (see Section \ref{sec:degree}.)

\subsection{Banach vector bundles\label{sec:banach_vc_bundles}}

We will use the language of Banach vector bundles in this paper, so we begin
by recalling basic definitions.

A \emph{vector bundle of rank }$n$\emph{\ over $\Lambda$} is a triple
$(E,\Lambda,\pi)$, where $E$ and $\Lambda$ are topological spaces and
$\pi:E\rightarrow\Lambda$ is a surjective continuous map satisfying the
following conditions

\begin{enumerate}
\item for all $\theta\in\Lambda$, the \emph{fiber} $E_{\theta}:=\pi
^{-1}(\theta)$ over $\theta$ is homeomorphic to $\mathbb{R}^{n}$,

\item for every $\theta\in\Lambda$ there exists an open neighborhood
$U_{\theta}$ of $\theta$ in $\Lambda$ and a homeomorphism
\[
\Phi_{\theta}:\pi^{-1}(U_{\theta})\rightarrow U_{\theta}\times\mathbb{R}^{n},
\]
called a \emph{local trivialization of $E$ over $U_{\theta}$}, such that

\begin{itemize}
\item $\pi_{0}\circ\Phi_{\theta}=\pi$, where $\pi_{0}:U_{\theta}%
\times\mathbb{R}^{n}\rightarrow U_{\theta}$ is the projection on $U_{\theta}$,

\item for every $\lambda\in U_{\theta}$ the restriction of $\Phi_{\theta}$ to
the fiber $E_{\lambda}$
\[
\Phi_{\theta}|_{E_{\lambda}}:E_{\lambda}\rightarrow\{\lambda\}\times
\mathbb{R}^{n}\simeq\mathbb{R}^{n}%
\]
is a vector space isomorphism. The set $U_{\theta}$ is called the \emph{base}
of the local trivialization $\Phi_{\theta}$.
\end{itemize}
\end{enumerate}

The space $E$ is called the \emph{total space} of the bundle, $\Lambda$ is
called its \emph{base}, and $\pi$ is its \emph{projection}.

When $\Phi_{\theta_{1}}:\pi^{-1}(U_{\theta_{1}})\rightarrow U_{\theta_{1}%
}\times\mathbb{R}^{n}$ and $\Phi_{\theta_{2}}:\pi^{-1}(U_{\theta_{2}%
})\rightarrow U_{\theta_{2}}\times\mathbb{R}^{n}$ are two local
trivializations of $E$ such that $U_{\theta_{1}}\cap U_{\theta_{2}}%
\neq\emptyset$, and $\lambda\in U_{\theta_{1}}\cap U_{\theta_{2}}$, the
function
\[
(\pi_{\mathbb{R}^{n}}\circ\Phi_{\theta_{2}}|_{E_{\lambda}})\circ
(\pi_{\mathbb{R}^{n}}\circ\Phi_{\theta_{1}}|_{E_{\lambda}})^{-1}%
:\mathbb{R}^{n}\rightarrow \mathbb{R}^n
\]
is called a \emph{transition function} between local trivializations.

Define the map%
\[
g_{\theta_{1},\theta_{2}}:U_{\theta_{1}}\cap U_{\theta_{2}}\rightarrow
GL\left(  n,\mathbb{R}\right)
\]
by the formula
\[
g_{\theta_{1},\theta_{2}}\left(  \lambda\right)  =(\pi_{\mathbb{R}^{n}}%
\circ\Phi_{\theta_{2}}|_{E_{\lambda}})\circ(\pi_{\mathbb{R}^{n}}\circ
\Phi_{\theta_{1}}|_{E_{\lambda}})^{-1}(\lambda),
\]
i. e. $g_{\theta_{1},\theta_{2}}$ sends each $\lambda$ the corresponding
transition function between local trivializations. The vector bundle
$(E,\Lambda,\pi)$ is called \emph{orientable} if
\[
\det\left(  g_{\theta_{1},\theta_{2}}\left(  \lambda\right)  \right)  >0
\]
for every $\theta_{1},$ $\theta_{2}$ and $\lambda\in U_{\theta_{1}}\cap
U_{\theta_{2}}.$

If we are given a vector bundle $\pi:E\rightarrow\Lambda$ with a fixed
collection of local trivializations $\{\Phi_{\theta}:\pi^{-1}(U_{\theta
})\rightarrow U_{\theta}\times \mathbb{R}^n\}$ whose bases form an open cover
$\mathcal{U}=\{U_{\theta}\}$ of $\Lambda$, then we call it a \emph{Banach
vector bundle} provided that all transition functions between local
trivializations with overlapping bases are isometries. If spaces $E$ and
$\Lambda$ are manifolds and $\pi$ is a smooth map, then we call it a
\emph{smooth Banach vector bundle}.

For Banach vector bundles we are able to introduce a meaningful notion of a
norm on fibers as follows. For every $v\in E$ such that $\pi(v)\in U_{\theta}%
$, where $U_{\theta}$ is trivialized by $\Phi_{\theta}$, we define
\[
\Vert v\Vert:=\Vert\pi_{\mathbb{R}^{n}}\circ\Phi_{\theta}(v)\Vert
_{\mathbb{R}^{n}},
\]
where $\Vert\cdot\Vert_{\mathbb{R}^{n}}$ is some norm on
$\mathbb{R}^{n}$. Since all transition functions between local trivializations
with overlapping bases are isometries, we see that $\Vert v\Vert$ does not
depend on the choice of $\theta$. We emphasize that the notation $\Vert
v\Vert$ should be always understood as the norm on the fiber $E_{\pi\left(
v\right)  }$, since $E$ is not a normed space and it makes no sense to talk
about the norm of an element of $E$.

Let $(E^{u},\Lambda,\pi^{u})$ and $(E^{s},\Lambda,\pi^{s})$ be smooth Banach
vector bundles of rank $u$ and $s$, respectively. A \emph{Whitney sum} of
$(E^{s},\Lambda,\pi^{s})$ and $(E^{u},\Lambda,\pi^{u})$ is a smooth Banach
vector bundle defined as follows. The total space is defined as%
\[
E:=E^{u}\oplus E^{s}:=\bigsqcup_{\theta\in\Lambda}E_{\theta}^{u}\oplus
E_{\theta}^{s},
\]
where $\oplus$ denotes the direct sum of vector spaces. In other words, the
elements of $E$ can be identified with triples $\left(  \theta;x,y\right)  $,
where $\theta\in\Lambda$ and $\left(  x,y\right)  \in$ $E_{\theta}^{u}\oplus
E_{\theta}^{s}$. The base space of the Whitney sum is $\Lambda$ and the
projection $\pi:E\rightarrow\Lambda$ is given by
\[
\pi\left(  \theta;x,y\right)  =\theta.
\]

Assume that $\left\{  \Phi_{\theta}^{u}:\left(  \pi^{u}\right)  ^{-1}\left(
U_{\theta}\right)  \rightarrow U_{\theta}\times\mathbb{R}^{u}\right\}  $ and
$\left\{  \Phi_{\theta}^{s}:\left(  \pi^{s}\right)  ^{-1}\left(  U_{\theta
}\right)  \rightarrow U_{\theta}\times\mathbb{R}^{s}\right\}  $ are
collections of local trivialisations for $(E^{u},\Lambda,\pi^{u})$ and
$(E^{s},\Lambda,\pi^{s})$, respectively. Let $W_{\theta}\subset\Lambda$ be an
open set which is contained in the domain of both $\Phi_{\theta}^{u}$ and
$\Phi_{\theta}^{s}.$ A local trivialization $\Phi_{\theta}:\pi^{-1}\left(
W_{\theta}\right)  \rightarrow W_{\theta}\times\mathbb{R}^{u+s}$ for the
Whitney sum can be defined as
\[
\Phi_{\theta}\left(  \lambda;x,y\right)  =\left(  \lambda,\left(
\pi_{\mathbb{R}^{u}}\circ\Phi_{\theta}^{u}\right)  \left(  x\right)  ,\left(
\pi_{\mathbb{R}^{s}}\circ\Phi_{\theta}^{s}\right)  \left(  y\right)  \right)
,
\]
where $\lambda\in W_{\theta},$ $x\in E_{\lambda}^{u}$ and $y\in E_{\lambda
}^{s}.$

\subsection{Alexander--Spanier cohomology theory}
\label{subsec:cohomology}
In our paper, we will work with the cohomology theory based on Alexander--Spanier cochains. An elementary exposition of this theory can be found in the book of Massey \cite{Massey1978}. Here, we only recall the basic facts and outline the differences between singular and Alexander--Spanier cohomology theories. In fact, they coincide on paracompact and locally contractible Hausdorff spaces (see Spanier \cite{SpanierAlgTop}, Ch. 6, Section 9).

The Alexander--Spanier cohomology theory satisfies the usual Eilenberg--Steenrod axioms for cohomology, i. e. exactness, homotopy, excision and dimension axioms. Just like in the singular theory, one can endow cohomology with the bilinear and natural cup product operation. We will be particularly interested in the relative cup product
$$\smile: H^\ast(X) \times H^*(X,A) \rightarrow H^\ast(X,A).$$
The naturality in the relative case means that 
$$f_{(X,A),(Y,B)}^\ast(u \smile v) = f^\ast_{X,Y}(u) \smile f^\ast_{(X,A),(Y,B)}(v)$$
for every continuous map $f: (X, A) \rightarrow (Y,B)$, $u \in H^\ast(Y)$ and $v \in H^\ast(Y,B).$

The excision axiom can be strengthened in the case of Alexander--Spanier cohomology theory, i. e. we have the following \textit{strong excision property}.

\begin{theorem}
[see Theorem 6.6.5 in \cite{SpanierAlgTop}]
\label{thm:str_excision} Assume that $X$ and $Y$ are paracompact Hausdorff spaces and let $A \subset X$ and $B \subset Y$ be closed subsets. If $\left(  X,A\right)  $ is a pair contained in
$\left(  Y,B\right)  $ and, moreover, $X\setminus A=Y\setminus B,$ then
inclusion induces an isomorphism%
\[
H^{\ast}\left(  Y,B\right)  \simeq H^{\ast}\left(  X,A\right).
\]
\end{theorem}

Alexander--Spanier cohomology theory has also a very useful \emph{tautness
property}, which we now recall. It will be used in the proof of the main theorem. 

Let $X$ be a locally compact Hausdorff space and let $A\subset X$. We denote by $\mathcal{N}(A)$ of all closed neighbourhoods of $A$, i.e. closed sets $B$ satisfying $A \subset \operatorname{int}B$.

A subfamily $\mathcal{M}\subset\mathcal{N}(A)$ is called \emph{cofinal} if for
every $\Omega\in\mathcal{N}(A)$ there exists $\Omega^{\prime}\in\mathcal{\mathcal{M}}$
such that $\Omega^{\prime}\subset \Omega.$

\begin{theorem}
[see Theorem 6.6.2 in \cite{SpanierAlgTop}]\label{thm:tautness} Suppose that $A$ is a closed subset of a paracompact Hausdorff space. If the subfamily $\mathcal{M}$ is cofinal in the family of all closed neigbhorhoods of $A$, then the pair 
$$\left(H^{\ast}%
(A),\{i_{A,\Omega}^{\ast}\}_{\Omega\in\mathcal{M}}\right)$$ is the direct limit of the direct system
$$\left\{(H^{\ast}(\Omega),\{i_{\Omega,\Omega^{\prime}}^{\ast}\}_{\Omega^{\prime
}\in\mathcal{M}})\right\}_{\Omega\in\mathcal{\mathcal{M}}},$$ which means that
\begin{itemize}
    \item for all $v\in H^{\ast}(A)$ there exists $\Omega \in \mathcal{M}$ and $v' \in H^{\ast}(\Omega)$ such that $i_{A,\Omega}^\ast(v')= v,$
    \item for all $\Omega \in \mathcal{M}$ and $v \in H^\ast(\Omega)$ the condition $i^\ast_{A,\Omega}(v)= 0$ implies that there exists $\Omega' \in \mathcal{M}$ such that $\Omega'\subset \Omega$ and $i^*_{\Omega',\Omega}(v) = 0.$
\end{itemize}
\end{theorem}

\subsection{Cohomology of vector bundles}
\label{subsec:cohomology_bundles}
We briefly recall here certain facts concerning the cohomology of vector bundles and their
so-called Thom spaces. 

Let $\left(  E\,,\Lambda,\pi\right)  $ be a smooth Banach vector bundle of rank
$n$ over a compact and connected manifold $\Lambda$. We define the sets%
\begin{align*}
D_E   &  =\left\{  \left(  \theta;x\right)  \in E\mid\left\vert
\left\vert x\right\vert \right\vert _{\mathbb{R}^{n}}\leq1\right\}  ,\\
\partial D_E   &  =\left\{  \left(  \theta;x\right)  \in
E\mid\left\vert \left\vert x\right\vert \right\vert _{\mathbb{R}^{n}%
}=1\right\}  ,
\end{align*}
where $\left\vert \left\vert \cdot\right\vert \right\vert _{\mathbb{R}^{n}}$
denotes the norm on the fiber (see Section \ref{sec:banach_vc_bundles}). The
pair $\left(  D_E  ,\partial D_E  \right)  $ is
called \textit{the Thom space} of $\left(  E\,,\Lambda,\pi\right).$ The following \textit{Thom isomorphism theorem} will be very important in the proof of our main result.

\begin{theorem}
[see Theorem 11.3 and Lemma 11.5 in \cite{BredonTopGeo}%
]\label{thm:thom_isom} Let $\Lambda$ be a connected closed manifold and let $\left(  E\,,\Lambda,\pi\right)  $ be a smooth Banach vector bundle of rank
$n$. Then there
exists a cohomology class $v_{T}$ $\in H^{n}\left(  D_E,\partial D_E\right)  $
such that the map%
\[
H^{k}\left(  \Lambda\right)  \ni \alpha\mapsto\pi_{D_E}^{\ast}\left(  \alpha\right)
\smile v_{T}\in H^{k+n}\left(  D_E,\partial D_E\right)  ,
\]
where $\smile:H^{\ast}\left(  D_E\right)  \times H^{\ast}\left(  D_E,\partial
D_E\right)  \rightarrow H^{\ast}\left(  D_E,\partial D_E\right)$ is the relative
cup product, is an isomorphism. The coefficient ring is arbitrary if $E$
is an orientable vector bundle or $\mathbb{Z}_{2}$ otherwise.

Moreover, for every $\theta\in\Lambda,$ the cohomology class $v_{T}|_{\theta
}=i_{\theta}^{\ast}\left(v_{T}\right)  \in H^{n}\left(  D_{E,\theta},\partial
D_{E,\theta}\right),$ where $D_{E,\theta} = D_E \cap E_\theta$, $\partial D_{E,\theta} = \partial D_E \cap E_\theta$ and $i_{\theta}:\left(  D_{E,\theta},\partial D_{E,\theta
}\right)  \rightarrow\left(  D_E,\partial D_E\right)$ is the inclusion map, is a
generator.
\end{theorem}

\subsection{Brouwer degree}
\label{sec:degree}
Assume that $\Lambda$ is a connected compact smooth manifold of dimension $c$. Recall that the homology group $H_c(\Lambda; R)$ is isomorphic to $R$ if $\Lambda$ is orientable. Thus if $f:\Lambda\rightarrow\Lambda$ is a continuous map and
$\Lambda$ is orientable, then the induced homomorphism%
\[
f_{\ast}:H_{c}\left(  \Lambda;R\right)  \rightarrow H_{c}\left( \Lambda;R\right)
\]
is given by $f_{\ast}\left(  x\right)  =\left(  \deg f\right)  x,$ where $\deg
f\in R$ is called the \textit{degree} \textit{of }$f$. In the non-orientable
case we can define $\deg_{2}f\in\mathbb{Z}_{2}$ in a similar way. For a
detailed discussion of the degree, see Hatcher \cite{MR1867354}. A more
geometric interpretation of a degree can be found in the book of Milnor
\cite{Milnor}. 

We will need two simple facts related to the degree. For the lack of a reference, we provide a simple proof of the second fact.

\begin{proposition}
[Proposition 7.1 in \cite{BredonTopGeo}]\label{thm:lin_isom_degree} If $A:\mathbb{R}%
^{n}\rightarrow\mathbb{R}^{n}$ is a linear isomorphism, then the map induced
by $A$ on one-point compactification on $\mathbb{R}^{n}$ (i. e. the unit
sphere $\mathbb{S}^{n}$) has degree equal to the sign of $\det A$.
\end{proposition}

\begin{proposition}
\label{prop:degree} Assume that $\Lambda$ is a compact, connected, orientable $c$-manifold without boundary and $\eta:\Lambda\rightarrow\Lambda$ is a map
with $\deg\eta\neq0$. Then $\eta$ induces an isomorphism on cohomology with coefficients in a field of non-zero characteristic. In the case when $\Lambda$ is non-orientable, the condition
$\deg_{2}\eta\neq0$ implies that $\eta$ induces an isomorphism on cohomology
with coefficients in $\mathbb{Z}_{2}$.
\end{proposition}
\begin{proof}
We will use the following standard facts from algebraic topology:
\begin{enumerate}
\item There exists a bilinear map%
\[
\frown:H^{p}\left(  \Lambda\right)  \times H_{n}\left(  \Lambda\right)  \rightarrow
H_{n-p}\left(  \Lambda\right)  ,
\]
called the \emph{cap product} such that if $f:\Lambda\rightarrow \Lambda'$ is a continuous
map, then
\begin{equation}
f_{\ast}\left(  f^{\ast}\left(  u\right)  \frown v\right)  =u\frown f_{\ast
}\left(  v\right)  \label{eq:cap_product_1}%
\end{equation}
for all $u\in H^{\ast}\left(  \Lambda\right)  $ and $v\in H_{\ast}\left(  \Lambda'\right)
$ (see Section 3.3. in \cite{MR1867354})
\item There exists a homology class $\left[  \Lambda\right]  \in H_{c}\left(
\Lambda\right)  $, called \emph{the fundamental class of }$\Lambda$, such that
the map $\Psi:H^{i}\left(  \Lambda\right)  \rightarrow H_{c-i}\left(
\Lambda\right)$ given by the formula
\[
\Psi\left(  u\right)  =u\frown\left[  \Lambda\right]
\]
is an isomorphism (this is the \textit{Poincar\'{e} duality theorem;} see, for
example, Theorem 3.30 in \cite{MR1867354}).

\item The cohomology groups of $\Lambda$ with field coefficients are finite
dimensional as vector spaces (see, for example, Proposition 5.3.1 in \cite{BottTu}).
\end{enumerate}

For a proof, let $d=\deg\eta$ or $\deg_{2}\eta$ in the nonorientable case. By the definition of the degree, we have
\begin{equation}
\eta_{\ast}\left(  \left[  \Lambda\right]  \right)  =d\left[
\Lambda\right]  . \label{eq:pf_prop_degree}%
\end{equation}

Let $u\in H^{i}\left(  \Lambda\right)$. Applying \eqref{eq:cap_product_1}, \eqref{eq:pf_prop_degree} and the bilinearity of the cap product successively, we obtain
\begin{align*}
\eta_{\ast}\left(  \Psi\left(  \eta^{\ast}\left(  u\right)  \right)  \right)
&  =\eta_{\ast}\left(  \eta^{\ast}\left(  u\right)  \frown\left[
\Lambda\right]  \right) \\
&  =u\frown\eta_{\ast}\left(  \left[  \Lambda\right]  \right) \\
&  =u\frown d\left[  \Lambda\right] \\
&  =d\left(  u\frown\left[  \Lambda\right]  \right) \\
&  =d\Psi\left(  u\right)  .
\end{align*}
Since $d$ is invertible by our choice of coefficients, we can rewrite
\[
\left(  d^{-1}\cdot\Phi^{-1}\circ\eta_{\ast}\circ\Phi\right)  \circ\eta^{\ast
}=\operatorname{id},
\]
which implies that $\left(  d^{-1}\cdot\Phi^{-1}\circ\eta_{\ast}\circ
\Phi\right)  $ is a left inverse of $\eta^{\ast}$. It follows that $\eta
^{\ast}$ is an injective linear map and in view of the third auxiliary fact
we conclude that it is an isomorphism as well.
\end{proof}

\subsection{Conley index for maps}

In this section we briefly recall the notions related to the Conley index for
maps developed in \cite{Mrozek1990}.  The
Conley index is defined there for homeomorphisms on locally compact metric spaces.
Nevertheless, it is indicated there that all the results are true for not
necessarily invertible maps.

\subsubsection{Leray reduction functor}

The Conley index is defined using the so-called \emph{Leray reduction
functor}. We outline this construction now. Fix a commutative ring with
identity $R$. Recall that a collection $E=(E_{n})_{n\in\mathbb{N}}$ is called
a graded $R$-module if each $E_{n}$ is an $R$-module. Likewise, we use the
notation $e:E\rightarrow E$ for a collection $(e_{n})_{n\in\mathbb{N}}$, where
$e_{n}:E_{n}\rightarrow E_{n}$ are endomorphisms and call $e$ an endomorphism
of $E$.

By $\mathcal{E}$ we denote a category whose objects are pairs $(E,e)$, where
$E$ is a graded $R$-module and $e:E\rightarrow E$ is an endomorphism, and a
collection $\varphi=(\varphi_{n})_{n\in\mathbb{N}}$ is a morphism between
$(E,e)$ and $(F,f)$ if each $\varphi_{n}:E_{n}\rightarrow F_{n}$ is a
homomorphism of modules such that the diagram
\begin{equation}%
\begin{array}
[c]{ccc}%
E_{n} & \overset{e_{n}}{\longrightarrow} & E_{n}\\
\varphi_{n}\downarrow\quad &  & \quad\downarrow\varphi_{n}\\
F_{n} & \underset{f_{n}}{\longrightarrow} & F_{n}%
\end{array}
\label{eq:cat-endo-mor}%
\end{equation}
commutes.

By $\mathcal{M}$ we denote a full subcategory of $\mathcal{E}$ whose objects
are pairs $(E,e)$ such that $e=(e_{n})_{n\in\mathbb{N}}$ is a monomorphism, i.
e. each $e_{n}$ is a monomorphism. Similarly, by $\mathcal{I}$ we mean a full
subcategory of $\mathcal{M}$ whose objects are pairs $(E,e)$ such that
$e=(e_{n})_{n\in\mathbb{N}}$ is an isomorphism, i. e. each $e_{n}$ is an isomorphism.

The Leray reduction functor will map the category $\mathcal{E}$ into its
subcategory $\mathcal{I}$. The definition goes as follows. The
\emph{generalized kernel} of $e$ is defined as
\[
\operatorname{gker}(e):=\bigcup\{e^{-n}(0)\mid n\in\mathbb{N}\}
\]
Since $e(\operatorname{gker}(e))\subset\operatorname{gker}(e)$, there is a
well-defined induced endomorphism
\[
e^{\prime}:E/\operatorname{gker}(e)\rightarrow E/\operatorname{gker}(e),
\]
given by
\begin{equation}
e^{\prime}\left(  [x]\right)  =[e(x)]. \label{eq:leray_induced_map_1}%
\end{equation}
Observe that $(E/\operatorname{gker}(e),e^{\prime})\in\mathcal{M}$. If
$\varphi\in\mathcal{E}\left(  (E,e),(F,f)\right)  $, then thanks to the
commutativity of the diagram (\ref{eq:cat-endo-mor}) we have $\varphi
(\operatorname{gker}(e))\subset\operatorname{gker}(f)$. It follows that
\[
\varphi^{\prime}:(E/\operatorname{gker}(e),e^{\prime})\rightarrow
(F/\operatorname{gker}(f),f^{\prime})
\]
given for each $n$ by
\[
\varphi_{n}^{\prime}([x])=[\varphi_{n}(x)]
\]
is a well-defined morphism in $\mathcal{M}\left(  (E/\operatorname{gker}%
(e),e^{\prime}),(F/\operatorname{gker}(f),f^{\prime})\right)  .$

The assignment
\[
\mathcal{LM}(E,e):=(E/\operatorname{gker}(e),e^{\prime})
\]
and
\[
\mathcal{LM}(\varphi):=\varphi^{\prime}%
\]
defines a covariant functor $\mathcal{LM}:\mathcal{E}\rightarrow\mathcal{M}$.

\begin{notation}
Throughout the paper, we will continue to use the notation $e^{\prime}$ for
the induced map (\ref{eq:leray_induced_map_1}) of $e$.
\end{notation}

Now let $(G,g)\in\mathcal{M}$. The \emph{generalized image} of $g$ is defined
as
\[
\operatorname{gim}(g):=\bigcap\{g^{n}(G)\mid n\in\mathbb{N}\}.
\]
Moreover, since $g(\operatorname{gim}(g))\subset\operatorname{gim}(g)$, we can
define the restriction
\[
g^{\prime\prime}:\operatorname{gim}{g}\rightarrow\operatorname{gim}{g}%
\]
by
\[
g^{\prime\prime}:=g|_{\operatorname{gim(g)}}.
\]
One can check that $(\operatorname{gim}(g),g^{\prime\prime})\in\mathcal{I}$
(see \cite{Mrozek1990} for details). Similarly, for a morphism $\xi
\in\mathcal{M}((G,g),(H,h))$, we have $\xi(\operatorname{gim}(g))\subset
\operatorname{gim}(h)$ because the diagram (\ref{eq:cat-endo-mor}) is
commutative. Therefore, we can define a morphism
\[
\xi^{\prime\prime}:(\operatorname{gim}(g),g^{\prime\prime})\rightarrow
(\operatorname{gim}(h),h^{\prime\prime})
\]
by the formula
\[
\xi_{n}^{\prime\prime}(x):=\xi_{n}(x).
\]
As before, the assignment
\[
\mathcal{LI}(G,g):=(\operatorname{gim}(g),g^{\prime\prime})
\]
and
\[
\mathcal{LI}(\xi):=\xi^{\prime\prime}%
\]
defines a covariant functor
\[
\mathcal{LI}:\mathcal{M}\rightarrow\mathcal{I}.
\]

\begin{definition}[see \cite{Mrozek1990}]
The composition $\mathcal{L}=\mathcal{LI}\circ\mathcal{LM}:\mathcal{E}%
\rightarrow$ $\mathcal{I}$ is called the Leray reduction functor.
\end{definition}

\begin{remark}
\label{rem:leray_funct}The key property of the Leray reduction functor is that
it acts as an identity functor on $\mathcal{I}$.
\end{remark}

We end this section with one more useful remark that will be useful in Section \ref{sec:ext_prod}.
\begin{remark}
    \label{rem:leray_funct_2} One can also consider functors $\mathcal{LM}$ and $\mathcal{LI}$ defined on the category of rings with a distinguished endomorphism. The construction is similar to the one we have just presented.  
\end{remark}

\subsubsection{Definition of the Conley index}
\label{subsubsec:con_ind_def}
Let $f:X\rightarrow X$ be a continuous map on a locally compact metric space
$X$. Assume that $J\subset\mathbb{Z}$ is an interval of integers containing
$0$. The sequence $\sigma:J\rightarrow X$ is called a \emph{solution} for $f$ through $x\in X$ if for all $k\in J$ we have $f(\sigma\left(  k\right)
)=\sigma\left(  k+1\right)  $ (if $k+1\in J)$ and $\sigma(0)=x$.

For a set $A\subset X$ we define the sets%
\begin{align*}
\operatorname{Inv}^{+}(A, f)  &  =\left\{  x\in A\mid f^{k}\left(  x\right)  \in
A\text{ for all }k\in\mathbb{N}\right\}  ,\\
\operatorname{Inv}^{-}(A, f)  &  =\left\{  x\in A\mid\text{there is a solution
}\sigma:\left\{  \ldots,-3,-2,-1,0\right\}  \rightarrow A\text{ through
}x\right\}
\end{align*}
\[
\operatorname{Inv}(A, f):=\operatorname{Inv}^{+}(A, f)\cap\operatorname{Inv}^-(A,f)
\]
called \emph{a forward invariant, a backward invariant and} \emph{an invariant
part} of $A$. The set $A\subset X$ is called \emph{invariant} if
$\operatorname{Inv}(A, f)=A.$ For a pair of subsets $A,B\subset X$ we say that $A$
is \emph{positively invariant} with respect to $B$ if $f(A)\cap B\subset A.$

Let $K$ be an invariant set. If there exists a compact set $N$ such that
$K=\operatorname{Inv}(N, f)$ and $K \subset \operatorname{int}N$, then we say that $K$ is an \emph{isolated invariant
set} and $N $ is called an \emph{isolating neighborhood} for $K$.

\begin{definition}
\label{def:ind_pair} An index pair $P=(P_{1},P_{2})$ for an isolated
invariant set $K$ with an isolating neighborhood $N$ is a pair of compact
subsets $P_{1},P_{2}$ of $N$ such that

\begin{enumerate}
\item $P_{1}$ and $P_{2}$ are positively invariant with respect to $N$,

\item $\operatorname{Inv}^{-}(N, f)\subset\operatorname{int}_{N}(P_{1})$ and
$\operatorname{Inv}^{+}(N,f)\subset N\setminus P_{2}.$

\item $P_{1}\setminus P_{2}  \subset
\operatorname{int}N\cap f^{-1}(\operatorname{int}N).$
\end{enumerate}
\end{definition}

The set of all index pairs with respect to $N$ will be denoted by
$\operatorname{IP}(N,f)$.

For $P=(P_{1},P_{2})\in\operatorname{IP}(N,f)$, we let
\begin{align*}
T(P)  &  :=(T(P_{1}),T(P_{2})):=(P_{1}\cup(X\setminus\operatorname{int}%
N),P_{2}\cup(X\setminus\operatorname{int}N)),\\
S(P)  &  :=(S(P_{1}),S(P_{2})):=(P_{1}\cup f(P_{2}),P_{2}\cup f(P_{2})).
\end{align*}

\begin{theorem}[Theorem 2.3 in \cite{Mrozek1990}]
We have the inclusions
\begin{equation}
f(P)\subset S(P)\subset T(P). \label{eq:T-f-inclusion}%
\end{equation}
Moreover, the induced maps $i_{P,S(P)}^{\ast}$, $i_{P,T(P)}^{\ast}$ are isomorphisms.
\end{theorem}

We are finally ready to recall the definition of the Conley index.

\begin{definition}
\label{def:index_map} An index map is a homomorphism $I_{P}:H^{\ast
}(P_{1},P_{2})\rightarrow H^{\ast}(P_{1},P_{2})$ defined by
\begin{equation}
I_{P}:=f_{P,T(P)}^{\ast}\circ i_{P,T(P)}^{\ast-1}. \label{eq:I_P-def}%
\end{equation}
The Conley index of the invariant set $K$ is defined as
\[
C(K,f):=\mathcal{L}(H^{\ast}(P_{1},P_{2}),I_{P}).
\]

\end{definition}

In \cite{Mrozek1990} it is proven that the index does not depend on the choice
of an isolating neighborhood or an index pair. We also have the following important remark.
\begin{remark}
\label{rem:con_ind_isom} If the index map $I_{P}$ is an isomorphism, then the Conley index
is isomorphic to $(H^{\ast}(P_{1},P_{2}),I_{P})$ (see Remark
\ref{rem:leray_funct}).
\end{remark}

\subsubsection{Auxiliary facts from the Conley index theory}
\label{sec:conley_index}
We collect here all supporting lemmas and
theorems regarding the Conley index that are used in our paper.

\begin{lemma}
[Theorem 4.4 in \cite{Mrozek1990}]\label{thm:index_pairs_existence} Let $N$ be an
isolating neighborhood for $K$. Then for every open set $W\supset K$ there
exists $P\in\mathrm{IP}(N)$ such that
\[
P_{1}\setminus P_{2}\subset W.
\]

\end{lemma}

\begin{definition}
[Definition 5.10 in \cite{Mrozek1990}]\label{def:related} Assume $P,Q\in \operatorname{IP}(N,f)$.
We say that $P$ is related to $Q$ if $P\subset Q$ and%
\[
\operatorname*{cl}(Q_{1}\backslash P_{2})\subset\operatorname*{int}N\cap
f^{-1}\left(  \operatorname{int}N\right)  .
\]

\end{definition}

\begin{lemma}
[Lemma 5.12 in \cite{Mrozek1990}\label{lem:cont_prop_1}]For every isolating
neighborhood $N$ there exist index pairs $P,R\in\operatorname{IP}(N,f)$ such that
$P\subset\operatorname{int}_{N}R$ and $P$ is related to $R$.
\end{lemma}

\begin{lemma}
[Lemma 5.15 in \cite{Mrozek1990}\label{lem:cont_prop_2}]Assume that
$P,R\in\operatorname{IP}(N,f)$, $P$ is related to $R$ and $P\subset
\operatorname{int}_{N}R$. Then there exists $Q\in\operatorname{IP}\left(
N,f\right)  $ such that $P\subset\operatorname{int}_{N}Q$, $Q\subset
\operatorname{int}_{N}R$, $P$ is related to $Q$ and $Q$ is related to $R$.
\end{lemma}

\begin{lemma}
[Lemma 7.5 in \cite{Mrozek1990}\label{lem:cont_prop_3}]Let $F:[0,1]\times
X\rightarrow X$ and denote $F_{\alpha}=F\left(  \alpha,\cdot\right)  $. Assume
that $N$ is an isolating neighborhood with respect to $F_{\alpha}$ for some
$\alpha\in\left[  0,1\right]  $. If $P,Q\in \operatorname{IP}\left(  N,F_{\alpha}\right)  $
are such that $P\subset\operatorname{int}_{N}Q$ and $P$ is related to $Q$,
then there exists a neighborhood $U$ of $\alpha$ in $\left[  0,1\right]  $
such that for every $\kappa\in U$ there exists $R\left(  \kappa\right)  \in
\operatorname{IP}\left(  N,\kappa\right)  $ satisfying $P\subset R\left(  \kappa\right)
\subset Q.$ Moreover, the inclusions $i:P\rightarrow R\left(  \kappa\right)
$, $\ j:R\left(  \kappa\right)  \rightarrow Q$ induce morphisms
\begin{align*}
i^{\ast}  &  :\left(  H^{\ast}\left(  R\left(  \kappa\right)  \right)
,I_{R\left(  \kappa\right)  }\right)  \rightarrow\left(  H^{\ast}\left(
P\right)  ,I_{P}\right)  ,\\
j^{\ast}  &  :\left(  H^{\ast}\left(  Q\right)  ,I_{Q}\right)  \rightarrow
\left(  H^{\ast}\left(  R\left(  \kappa\right)  \right)  ,I_{R\left(
\kappa\right)  }\right)
\end{align*}
in the category $\mathcal{E}$.
\end{lemma}

\begin{lemma}
[Theorem 6.4 in \cite{Mrozek1990}]\label{lem:inclusion_isom_ind_pairs} Assume
that $f:X\rightarrow X$ is a continuous map, $N$ is an isolating neighborhood
for $f$. If $P,Q\,\in\mathrm{IP}(N,f)$ satisfy $P\subset Q$, then the
inclusion $i_{P,Q}$ induces an isomorphism%
\[
\mathcal{L}\left(  i_{P,Q}^{\ast}\right)  :\mathcal{L}\left(  H^{\ast}\left(
Q_{1},Q_{2}\right)  ,I_{Q}\right)  \rightarrow\mathcal{L}\left(  H^{\ast
}\left(  P_{1},P_{2}\right)  ,I_{P}\right)
\]
of graded $R$-modules.
\end{lemma}

The next fact seems to be folklore in the Conley index
theory, but we provide a short proof for convenience of the reader.

\begin{proposition}
\label{prop: lim_ind_pairs} Let $K$ be an isolated invariant set with respect to $f$ and let $N$ be its isolating neighborhood. The family $$\left\{P_1: P=(P_1,P_2) \in \operatorname{IP}(M,f) \; \text{and}\; M\subset N \; \text{is an isolating neighborhood for $K$} \right\}$$ is cofinal (see Section \ref{subsec:cohomology}) in the family of all closed neighborhoods of $K.$
\end{proposition}
\begin{proof} 
Choose a closed neighborhood $\Omega$ of $K$. The compact set $\Omega^{\prime
}:=N\cap\Omega$ is an isolating neighborhood for $K$ as
well. Lemma \ref{thm:index_pairs_existence} 
implies that $\operatorname{IP}(\Omega^{\prime
},f) \neq \emptyset.$ Let $P\in\operatorname{IP}(\Omega',f)$ and observe that it satisfies $P_1\subset \Omega'\subset\Omega.$
\end{proof}


\section{External product for the Conley index}
\label{sec:ext_prod}

In this section, we give a construction of an external product for the Conley
index for maps on Banach vector bundles. This will be the main tool in our
proof of Theorem \ref{thm:cont_thm}.

Similar constructions have appeared, for example, in \cite{FloerConInd} for
the Conley index for flows, in \cite{srzednicki2000structure} for the Conley
index over a base and in \cite{Szybowski2007} for the Conley index over a
phase space. However, we are interested in obtaining the relationship between the Conley index and the cohomology ring of the base. Thus we will endow the index of the
invariant set with the structure of a left module over some subring of
$H^{\ast}(\Lambda)$. 

Let $(E,\Lambda,\pi)$ be a smooth Banach vector bundle with a compact base
space $\Lambda$ and let $f:E\rightarrow E$ be a continuous map.

For simplicity, let us denote
\begin{equation}
\tilde{f}:=\pi\circ {f_\Lambda}. \label{eq:def_f_lambda}
\end{equation}
(Here, $f_\Lambda = f\vert_{\Lambda}$ according to the notation introduced in Section \ref{subsec:notations}.) There is a particularly nice relationship between the maps $f$ and $\tilde{f}$.

\begin{proposition}
\label{def:admissible-map} With notations as above, we have
\begin{equation}
\tilde{f}\circ\pi\simeq\pi\circ f. \label{eq:admissibility}
\end{equation}
\end{proposition}
\begin{proof}
    Note that $$\tilde{f} \circ \pi = \pi \circ f_\Lambda \circ \pi = \pi \circ f \circ i_{\Lambda,E} \circ \pi.$$ The assertion follows from the fact that $i_{\Lambda,E} \circ \pi$ is homotopic to the identity via the homotopy $H: [0, 1]\times E \rightarrow E$ defined by the formula $$H(t,(\theta;z)) = \left(\theta;(1-t)z\right).$$
    This completes the proof.
\end{proof}

Note that $\tilde{f}(\Lambda)\subset\Lambda$, so there is an induced homomorphism
on cohomology
\[
\tilde{f}^{\ast}:H^{\ast}(\Lambda)\rightarrow H^{\ast}(\Lambda).
\]
We now pay attention to a useful consequence of Proposition \ref{lem:admissibility}.
\begin{lemma}
\label{lem:admissibility} Assume that $K$ is an isolated invariant set for the
map $f$ with an isolating neighborhood $N$ and let $P=(P_{1},P_{2}%
)\in\operatorname{IP}(N,f)$. Then
\begin{equation*}
\pi_{P_1}^{\ast}\circ \tilde{f}^{\ast}=f_{P_{1},T\left(
P_{1}\right)  }^{\ast}\circ\pi_{T\left(  P_{1}\right)}^{\ast}.
\end{equation*}

\end{lemma}
\begin{proof}
From (\ref{eq:T-f-inclusion}) we know that
\begin{equation}
f\left(  P_{1}\right)  \subset T\left(  P_{1}\right)  .
\label{eq:T-f-inclusion-1}%
\end{equation}
Using (\ref{eq:admissibility}) and (\ref{eq:T-f-inclusion-1}) we obtain 
\[
\tilde{f}\circ\pi_{{P_{1}}}=\left(  \tilde{f}\circ\pi\right)_{P_{1}}\simeq(\pi\circ
f)_{P_{1}}=\pi_{T(P_{1})}\circ f_{P_{1},T(P_1)}.
\]
The proof now follows from the homotopy property of cohomology.
\end{proof}
\begin{definition}
Let $K$ be an isolated invariant set for the map $f$ with an isolating
neighborhood $N$ and let $P=(P_{1},P_{2})\in\operatorname{IP}(N,f)$. We
define the external product
\[
\ast: \operatorname{gim}(\tilde{f}^\ast)\times \operatorname{gim}(I_P')\rightarrow \operatorname{gim}(I_P')
\]
by the formula
\begin{equation}
\alpha\ast\lbrack v]=\left[  \pi_{P_{1}}^{\ast}(\alpha)\smile v\right]  ,
\label{eq:def_ext_product}%
\end{equation}
where $\smile:H^{\ast}\left(  P_{1}\right)  \times H^{\ast}\left(  P_{1}%
,P_{2}\right) \rightarrow H^\ast(P_1,P_2) $ is the relative cup product.
\end{definition}
\begin{proposition}
The external product \eqref{eq:def_ext_product} is well-defined, i. e. if $[v_1] = [v_2]$, then 
\begin{equation}
    \left[  \pi_{P_{1}}^{\ast}(\alpha)\smile v_1\right] = \left[  \pi_{P_{1}}^{\ast}(\alpha)\smile v_2\right]
\end{equation}
and $[v] \in \operatorname{gim}(I_P')$ implies that $\alpha \ast [v] \in \operatorname{gim}(I_P').$
\end{proposition}
\begin{proof}
It suffices to show that
\begin{enumerate}
\item[a)] if $v\in\operatorname{gker}(I_{P})$, then $\pi_{P_{1}}%
^{\ast}(\alpha)\smile v\in\operatorname{gker}(I_{P})$,

\item[b)] if $[v]\in\operatorname{gim}(I_{P}^{\prime})$, then $\left[
\pi_{P_{1}}^{\ast}(\alpha)\smile v\right]  \in\operatorname{gim}%
(I_{P}^{\prime})$.
\end{enumerate}

We will start by proving that
\begin{equation}
\label{eq:pf_ext_prod_well_defined}
I_{P}^{k}\left(  \pi_{P_{1}}^{\ast}(\alpha)\smile v\right)  =\pi
_{P_{1}}^{\ast}\left(  \tilde{f}^{\ast k}\left(  \alpha\right)  \right)
\smile I_{P}^{k}(v)
\end{equation}
for every $k \in \mathbb{N}.$ First, consider the case $k = 1$. Observe that $$\pi_{{P_{1}}}=\pi_{{T(P_{1})}
}\circ i_{P_{1},T\left(  P_{1}\right)},$$ which gives
\begin{equation}
\pi_{P_{1}}^{\ast}=i_{P_{1},T\left(  P_{1}\right)  }^{\ast}\circ
\pi_{T(P_{1})}^{\ast}\label{eq:temp-projection-inclusion-TP1}.
\end{equation}
Now we compute
\begin{align*}
&  I_{P}\left(  \pi_{P_{1}}^{\ast}(\alpha)\smile v\right) =\left(f_{P,\,T(P)}^{\ast}\circ i_{P,\,T(P)}^{\ast-1}\right)  \left(
\pi_{P_{1}}^{\ast}(\alpha)\smile v\right)  \text{ }\nonumber\\
&  =\left(  f_{P,\,T(P)}^{\ast}\circ i_{P,\,T(P)}^{\ast-1}\right)  \left(
\left(  i_{P_{1},T\left(  P_{1}\right)  }^{\ast}\circ\pi_{T(P_{1})
}^{\ast}\right)  (\alpha)\smile v\right) \nonumber\\
&  =\left(  f_{P,\,T(P)}^{\ast}\circ i_{P,\,T(P)}^{\ast-1}\right)  \left(
\left(  i_{P_{1},T\left(  P_{1}\right)  }^{\ast}\circ\pi_{T(P_{1})
}^{\ast}\right)  (\alpha)\smile\left(  i_{P,\,T(P)}^{\ast}\circ i_{P,\,T(P)}%
^{\ast-1}\right)  \left(  v\right)  \right)  \text{ }\nonumber\\
&  =\left(  f_{P,\,T(P)}^{\ast}\circ i_{P,\,T(P)}^{\ast-1}\circ i_{P,\,T(P)}%
^{\ast}\right)  \left(  \pi_{T(P_{1}),\Lambda}^{\ast}(\alpha)\smile i_{P,\,T(P)}%
^{\ast-1}\left(  v\right)  \right) \nonumber\\
&  =f_{P,\,T(P)}^{\ast}\left(  \pi_{T(P_{1})}^{\ast}(\alpha)\smile
i_{P,\,T(P)}^{\ast-1}\left(  v\right)  \right)  \nonumber\\
&  =\left(  f_{P_{1},T\left(  P_{1}\right)  }^{\ast}\circ\pi_{T\left(
P_{1}\right)}^{\ast}\right)  \left(  \alpha\right)  \smile\left(
f_{P,\,T(P)}^{\ast}\circ\text{ }i_{P,\,T(P)}^{\ast-1}\right)  \left(
v\right) \nonumber\\
&  =\pi_{P_{1}}^{\ast}\left(  \tilde{f}^{\ast}\left(  \alpha\right)
\right)  \smile I_{P}\left(  v\right).\nonumber
\end{align*}
Here, we have used \eqref{eq:temp-projection-inclusion-TP1}, naturality of the relative cup product and Lemma \ref{lem:admissibility}. The case $k > 1$ follows by induction.

For the proof of (a), we first note that if $v\in\operatorname{gker}(I_{P}),$
then $I_{P}^{k}(v)=0$ for some $k\in\mathbb{N}$. Using formula \eqref{eq:pf_ext_prod_well_defined}, we obtain
\[
I_{P}^{k}\left(  \pi_{P_{1}}^{\ast}(\alpha)\smile v\right)  =\pi
_{P_{1}}^{\ast}\left(  \tilde{f}^{\ast k}\left(  \alpha\right)  \right)
\smile I_{P}^{k}(v)=\pi_{P_{1}}^{\ast}\left(  \tilde{f}^{\ast
k}\left(  \alpha\right)  \right)  \smile0=0,
\]
in view of bilinearity of the cup product. This completes the proof of (a).

To prove (b), let $[v]\in\operatorname{gim}(I_{P}^{\prime})$ and $m \in \mathbb{N}$. Then $[v]=(I_{P}^{\prime
})^{m}([w])$ for some $w\in H^{\ast}\left(  P_{1},P_{2}\right) $ and we have
\[
v-I_{P}^{m}(w)\in\operatorname{gker}(I_{P})
\]
and from (a) we conclude that
\[
\pi_{P_{1}}^{\ast}(\alpha)\smile\left(  v-I_{P}^{m}(w)\right)
\in\operatorname{gker}(I_{P}).
\]
Consequently,
\[
\pi_{P_{1}}^{\ast}(\alpha)\smile v-\pi_{P_{1}}^{\ast}(\alpha)\smile
I_{P}^{m}(w)=\pi_{P_{1}}^{\ast}(\alpha)\smile\left(  v-I_{P}^{m}(w)\right)
\in\operatorname{gker}(I_{P}).
\]
Hence
\begin{equation}
\lbrack\pi_{P_{1}}^{\ast}(\alpha)\smile v]=[\pi_{P_{1}}^{\ast
}(\alpha)\smile I_{P}^{m}(w)]. \label{eq:pf_ext_prod_1}%
\end{equation}

Since $\alpha\in$ $\operatorname{gim}(\tilde{f}^{\ast}),$ 
there exists $p=p(m)\in H^{\ast}(\Lambda)$ such that $\alpha=\tilde{f}^{\ast
m}(p)$. We apply \eqref{eq:pf_ext_prod_well_defined} to obtain%
\[
\left(I_{P}'\right)^{m}\left[  \pi_{P_{1}
}^{\ast}\left(  p\right)  \smile w\right] = \left[  (\pi
_{P_{1}}^{\ast}\circ \tilde{f}^{\ast m})\left(  p\right)  \smile
I_{P}^{m}\left(  w\right)  \right] = 
\left[  \pi_{P_{1}}^{\ast}(\alpha)\smile v\right],
\]
as required. This concludes our proof.
\end{proof}

In view of Definition \ref{def:index_map} of the Conley index, we could regard
the formula (\ref{eq:def_ext_product}) as an external product on the index. However, we have to prove that the external product we
have just defined does not depend on the choice of an index pair or an
isolating neighborhood for an invariant set.

\begin{remark}
Due to the properties of the cup product, the generalized image
$\operatorname{gim}(I_{P}^{\prime})$ with the external product
\eqref{eq:def_ext_product} is a graded module over $\operatorname{gim}%
(\tilde{f}^{\ast})$.
\end{remark}

Thus we should prove that two module structures defined using different index pairs (possibly with respect to different isolating neighborhoods) for the same isolated invariant set are isomorphic. This is the content of the following theorem.

\begin{theorem}
\label{thm:ind_ext_prod} Assume that $M$ and $N$ are isolating neighborhoods
for the same isolated invariant set for a map $f$ and let
$P\in\operatorname{IP}\left(  M,f\right)  $ and $Q\in\operatorname{IP}\left(
N,f\right)  .$ Then there exists a bijection%
\[
i:\operatorname{gim}(I_{Q}^{\prime})\rightarrow\operatorname{gim}%
(I_{P}^{\prime})
\]
satisfying%
\begin{align}
i\left(  \left[  v_{1}+v_{2}\right]  \right)   &  =i\left(  \left[
v_{1}\right]  \right)  +i\left(  \left[  v_{2}\right]  \right)  \text{ \qquad
for all }v_{1},v_{2}\in\operatorname{gim}(I_{Q}^{\prime}), \label{eq:isom_mod}%
\\
i\left(  \alpha\ast\left[  v\right]  \right)   &  =\alpha\ast i\left(  \left[  v\right]
\right)  \text{ \qquad for all }\alpha\in\operatorname{gim}(\tilde{f}^{\ast
})\text{ and }v\in\operatorname{gim}(I_{Q}^{\prime}). \label{eq:isom_mod_}%
\end{align}

\end{theorem}

In other words, Theorem \ref{thm:ind_ext_prod} says that $\operatorname{gim}%
(I_{P}^{\prime})$ and $\operatorname{gim}(I_{Q}^{\prime})$ are isomorphic as
left modules over $\operatorname{gim}(\tilde{f}^{\ast}).$ We will divide the
proof of Theorem \ref{thm:ind_ext_prod} into several lemmas, in which we will
adapt certain ideas from \cite{FloerConInd} and \cite{Mrozek1990} to our
setting. We begin with a definition.

\begin{definition}
\label{def:dominates} Let $N$ and $M$ be isolating neighborhoods for maps $f$ and $g$, let $P=(P_{1},P_{2})\in\operatorname{IP}(N,f)$
and $Q=(Q_{1},Q_{2})\in\operatorname{IP}(M,g)$. We will write $P \preceq Q$ if there exists a continuous
map
\[
\varphi:(P_{1},P_{2})\rightarrow(Q_{1},Q_{2})
\]
such that 
\begin{itemize}
    \item $\varphi^*$ is a morphism in the category $\mathcal{E}$, i.e. \begin{equation} \varphi^* \circ I_Q = I_P \circ \varphi^*, \label{eq:def_dominates0}
    \end{equation}
    \item the homomorphism \begin{equation}
\mathcal{L}(\varphi^{\ast}):\mathcal{L}\left(  H^{\ast}(Q_{1},Q_{2}%
),I_{Q}\right)  \rightarrow\mathcal{L}\left(  H^{\ast}(P_{1},P_{2}%
),I_{P}\right)  \label{eq:def_dominates1}%
\end{equation}
is an isomorphism of graded $R$-modules,
\item we have the equality
\begin{equation}
\pi_{Q_{1}}\circ\;\varphi\simeq\pi_{P_{1}}. \label{eq:def_dominates2}%
\end{equation}
\end{itemize}
\end{definition}

Note that condition \eqref{eq:def_dominates1} from Definition
\ref{def:dominates} implies that the Conley indices of $\operatorname{Inv}(N,f)$
and $\operatorname{Inv}(M,g)$ coincide. The condition \eqref{eq:def_dominates2} ensures that $\mathcal{L}(\varphi^*)$ preserves the external product.

\begin{lemma}
\label{lem:dominates_isom} Suppose that $f$ and $g$ are homotopic
maps. If $P\in\operatorname{IP}(N,f),$ $Q\in\operatorname{IP}(M,g)$ and $Q \preceq P$, then there exists a bijection $i:\operatorname{gim}%
(I_{Q}^{\prime})\rightarrow\operatorname{gim}(I_{P}^{\prime})$ satisfying
(\ref{eq:isom_mod}) and (\ref{eq:isom_mod_}).
\end{lemma}

\begin{proof}
Note that $\operatorname{gim}(\tilde{f}^{\ast})\simeq\operatorname{gim}%
(g_{\Lambda}^{\ast}),$ because $f$ and $g$ are homotopic and therefore their
induced maps on cohomology coincide.

We will show that we can take $i=\mathcal{L}(\varphi^{\ast})$. Since
$\mathcal{L}(\varphi^{\ast})$ is an isomorphism of graded $R$-modules, where
$R$ is the coefficient ring for cohomology, it preserves addition, i. e. we
have
\[
\mathcal{L}(\varphi^{\ast})\left(  \left[  v_{1}+v_{2}\right]  \right)
=\mathcal{L}(\varphi^{\ast})\left(  \left[  v_{1}\right]  \right)
+\mathcal{L}(\varphi^{\ast})\left(  \left[  v_{2}\right]  \right)
\]
for all $v_{1},v_{2}\in\operatorname{gim}(I_{Q}^{\prime})=\mathcal{L}\left(
H^{\ast}(Q_{1},Q_{2}),I_{Q}\right)  $ and the condition (\ref{eq:isom_mod}) is
met. To show that the condition (\ref{eq:isom_mod_}) is satisfied, let
$\alpha\in\operatorname{gim}(\tilde{f}^{\ast})$ and $[v]\in\operatorname{gim}%
(I_{Q}^{\prime})$. By the definition of the Leray functor and \eqref{eq:def_dominates2}, we have
\begin{align*}
\mathcal{L}(\varphi^{\ast})\left(  \alpha\ast\lbrack v]\right)   &  =\mathcal{L}%
(\varphi^{\ast})\left(  \left[  \pi_{Q_{1}}^{\ast}(\alpha)\smile v\right]
\right)  =\left[  \varphi^{\ast}(\pi_{Q_{1}}^{\ast}(\alpha)\smile v)\right] \\
&  =\left[  \varphi^{\ast}(\pi_{Q_{1}}^{\ast}(\alpha))\smile\varphi^{\ast
}(v)\right] =\left[  (\pi_{Q_{1}}\circ\varphi)^{\ast}(\alpha)\smile
\varphi^{\ast}(v)\right]  \text{ }\\
&  =\left[  \pi_{P_{1}}^{\ast}\left(\alpha\right)  \smile\varphi^{\ast
}(v)\right] =\alpha\ast\lbrack\varphi^{\ast}(v)]=\alpha\ast\mathcal{L}(\varphi^{\ast})\left(  [v]\right),
\end{align*}
which proves (\ref{eq:isom_mod_}). This completes the proof.
\end{proof}

To prove Theorem \ref{thm:ind_ext_prod}, we will show that for any index pairs
$P\,$\ and $Q\,$, there exists an index pair $R$ such that $R \preceq P$
and $R \preceq Q$ and then use Lemma \ref{lem:dominates_isom}.

\begin{lemma}
\label{lem:inclusion_dominates} If $P,Q\in\operatorname{IP}(N,f)$ are index
pairs such that $P\subset Q$, then $P \preceq Q$.
\end{lemma}
\begin{proof} The
equalities \eqref{eq:def_dominates0} and \eqref{eq:def_dominates2} follow from direct computation.
The fact that inclusion induces an isomorphism (\ref{eq:def_dominates1}) is
the content of Lemma \ref{lem:inclusion_isom_ind_pairs}.
\end{proof}

\begin{lemma}
\label{lem:domination_1} If $P,Q\in\operatorname{IP}(N,f)$, then there exists
an index pair $T\in\operatorname{IP}(N,f)$ such that $T \preceq P$ and $T \preceq Q$. In particular, there is an isomorphism $i:\operatorname{gim}%
(I_{Q}^{\prime})\rightarrow\operatorname{gim}(I_{P}^{\prime})$ satisfying
\eqref{eq:isom_mod} and \eqref{eq:isom_mod_}.
\end{lemma}
\begin{proof}
Let $T:=P\cap Q$. It follows directly from Definition \ref{def:ind_pair} of
and index pair that $T\in\operatorname{IP}(N,f)$. Since $T\subset P$ and
$T\subset Q $, the proof is completed by an application of Lemma
\ref{lem:inclusion_dominates}.
\end{proof}

We are now ready to give a proof of Theorem \ref{thm:ind_ext_prod}. \\
\begin{proof}
[Proof of Theorem \ref{thm:ind_ext_prod}]Assume first that $M\subset N$.
Combining the second and third condition Definition \ref{def:ind_pair} of an index pair, we observe that $W:=\operatorname{int}M\cap f^{-1}(\operatorname{int}M)$ is a
neigbhorhood of $\operatorname{Inv}M=\operatorname{Inv}N$. Thus by Lemma
\ref{thm:index_pairs_existence} we infer that there exists an index pair
$R=(R_{1},R_{2})\in\operatorname{IP}(N,f)$ satisfying%
\[
R_{1}\backslash R_{2}\subset\operatorname{int}M\cap f^{-1}(\operatorname{int}%
M).
\]

Let $S_{1}:=R_{1}\cap M$ and $S_{2}:=R_{2}\cap M$. Then $S=(S_{1},S_{2}%
)\in\operatorname{IP}(M,f)$ by a direct verification of conditions of Definition
\ref{def:ind_pair}. Moreover, using the assumption that $M\subset N$ we
conclude that $S\subset R$. Since $R_{1}\backslash R_{2}=S_{1}\backslash
S_{2},$ the strong excision property of Alexander--Spanier cohomology (see
Theorem \ref{thm:str_excision}) implies that the inclusion $S\subset R$
induces an isomorphism in cohomology. Therefore conditions
(\ref{eq:def_dominates1}) and (\ref{eq:def_dominates2}) are met and $S \preceq R$. Then Lemma \ref{lem:dominates_isom} that $\operatorname{gim}%
(I_{R}^{\prime})$ and $\operatorname{gim}(I_{S}^{\prime})$ are isomorphic as
left modules over $\operatorname{gim}(\tilde{f}^{\ast}).$

Because $Q,R\in\operatorname{IP}(N,f)$, Lemma implies that
\ref{lem:domination_1} there is an isomorphism $\operatorname{gim}%
(I_{Q}^{\prime})\simeq\operatorname{gim}(I_{R}^{\prime})$ of left $\operatorname{gim}(\tilde{f}^\ast)$-modules. Similarly, there is an isomorphism
$\operatorname{gim}(I_{P}^{\prime})\simeq\operatorname{gim}(I_{S}^{\prime}).$
It follows that $\operatorname{gim}(I_{P}^{\prime})\simeq\operatorname{gim}%
(I_{Q}^{\prime})$ and the case $M\subset N$ is settled.

The general case follows from considering the isolating neighborhood $M\cap
N.$ The proof is now complete.
\end{proof}

If $K$ is an isolated invariant set, we will denote%
\[
C_{\Lambda}\left(  K,f\right)  :=\operatorname{gim}(I_{P}^{\prime})
\]
considered as a left module over $\operatorname{gim}{\tilde{f}^\ast}$ with
multiplication defined by the external product (\ref{eq:def_ext_product}). In
view of Theorem \ref{thm:ind_ext_prod}, the above module structure does not
depend on the choice of an isolating neighborhood $N$ for $K.$

Now we would like to prove that the classical continuation property of the
Conley index extends to our setting. For this purpose, we introduce the
following definition.

\begin{definition}
Assume that $F:\left[  0,1\right]  \times E\rightarrow E$ is a continuous
family of maps. Let $N\subset E$ be an isolating neighborhood with
respect to $F\left(  \alpha,\cdot\right)  $ for all $\alpha\in\left[
0,1\right]  $\thinspace$\ $and let $K_{0}=\operatorname{Inv}\left(  N,F\left(
0,\cdot\right)  \right)  $ and $K_{1}=\operatorname{Inv}\left(  N,F\left(
1,\cdot\right)  \right).$ Then we say that $K_{0}$ and $K_{1}$ are
related by continuation.
\end{definition}
For convenience, we will denote $F_{\alpha}=F\left(  \alpha
,\cdot\right).$ We have the following theorem.
\begin{theorem}
\label{thm:cont_prop} If $K_{0}$ and $K_{1}$ are related by continuation, then
$C_{\Lambda}\left(  K_{0},F_{0}\right)  \simeq C_{\Lambda}\left(  K_{1}%
,F_{1}\right)  .$
\end{theorem}
\begin{proof}
Our proof is a minor modification of an argument for the Conley index given in
\cite{Mrozek1990}. Nevertheless, we repeat the argument for the convenience of the reader. 

Let $K_{\alpha}:=\operatorname{Inv}\left(  N,F_{\alpha}\right)  .$ Since
$\left[  0,1\right]  $ is connected, it suffices to show that for every
$\alpha\in\left[  0,1\right]  $ there exists a neighborhood $V$ of $\alpha$
such that for all $\beta\in V$ we have $C_{\Lambda}\left(  K_{\alpha
},F_{\alpha}\right)  \simeq$ $C_{\Lambda}\left(  K_{\beta},F_{\beta}\right)
.$

Let us fix $\alpha\in\left[  0,1\right]  $. Using Lemma \ref{lem:cont_prop_1}, we can choose index pairs $P,R\in \operatorname{IP}\left(  N,F_{\alpha
}\right)  $ such that $P\subset\operatorname{int}_{N}R$ and $P$ is related to
$R$. Then Lemma \ref{lem:cont_prop_2} implies that there exists an index pair
$Q\in \operatorname{IP}\left(  N,F_{\alpha}\right)  $ such that $P\subset\operatorname{int}%
_{N}Q,$ $Q\subset\operatorname{int}_{N}R,$ $P$ is related to $Q$ and $Q$ is
related to $R.$

An application of Lemma \ref{lem:cont_prop_3} for $P$ and $Q$ yields a
neighborhood $U_{1}$ of $\alpha$ such that for every $\beta\in U_{1}$ there
exists $S\left(  \beta\right)  \in \operatorname{IP}\left(  K_{\beta},F_{\beta}\right)  $
satisfying $P\subset S\left(  \beta\right)  \subset Q$. Moreover, there are
morphisms%
\begin{align*}
i_{1}^{\ast}  &  :\left(  H^{\ast}\left(  S\left(  \beta\right)  \right)
,I_{S\left(  \beta\right)  }\right)  \rightarrow\left(  H^{\ast}\left(
P\right)  ,I_{P}\right)  ,\\
j_{1}^{\ast}  &  :\left(  H^{\ast}\left(  Q\right)  ,I_{Q}\right)
\rightarrow\left(  H^{\ast}\left(  S\left(  \beta\right)  \right)
,I_{S\left(  \beta\right)  }\right)
\end{align*}
induced by inclusions $i_{1}:P\rightarrow S\left(  \beta\right)  $,
$\ j_{1}:S\left(  \beta\right)  \rightarrow Q.$

Similarly, we can use Lemma \ref{lem:cont_prop_3} for $Q$ and $R$ and obtain a
neighborhood $U_{2}$ of $\alpha$ such that for every $\beta\in U_{2}$ there
exists $T\left(  \beta\right)  \in \operatorname{IP}\left(  K_{\beta},F_{\beta}\right)  $
satisfying $Q\subset T\left(  \beta\right)  \subset R$ and there are morphisms%
\begin{align*}
i_{2}^{\ast}  &  :\left(  H^{\ast}\left(  T\left(  \beta\right)  \right)
,I_{T\left(  \beta\right)  }\right)  \rightarrow\left(  H^{\ast}\left(
Q\right)  ,I_{Q}\right)  ,\\
j_{2}^{\ast}  &  :\left(  H^{\ast}\left(  R\right)  ,I_{R}\right)
\rightarrow\left(  H^{\ast}\left(  T\left(  \beta\right)  \right)
,I_{T\left(  \beta\right)  }\right)  ,
\end{align*}
where $i_{2}:Q\rightarrow T\left(  \beta\right)  $ and $j_{2}:T\left(
\beta\right)  \rightarrow R$ are inclusions.

Choosing $\beta\in U_{1}\cap U_{2},$ we therefore obtain a commutative diagram
of morphisms induced by inclusions%
\begin{equation}%
\begin{array}
[c]{ccc}%
\left(  H^{\ast}\left(  S\left(  \beta\right)  \right)  ,I_{S\left(
\beta\right)  }\right)  & \overset{j_{1}^{\ast}}{\longleftarrow} & \left(
H^{\ast}\left(  Q\right)  ,I_{Q}\right) \\
i_{1}^{\ast}\downarrow\quad &  & \quad\uparrow i_{2}^{\ast}\\
\left(  H^{\ast}\left(  P\right)  ,I_{P}\right)  & \underset{j^{\ast}%
}{\longleftarrow} & \left(  H^{\ast}\left(  T\left(  \beta\right)  \right)
,I_{T\left(  \beta\right)  }\right)  ,
\end{array}
\label{eq:cont_prop_pf_diag}%
\end{equation}
where $j=i_{2}\circ j_{1}\circ i_{1}:P\rightarrow T\left(  \beta\right)  .$
Since $S\left(  \beta\right)  ,T\left(  \beta\right)  \in\operatorname{IP}%
\left(  K_{\beta},F_{\beta}\right)  ,$ Lemma
\ref{lem:inclusion_isom_ind_pairs} implies that $\mathcal{L}\left(
j_{1}^{\ast}\circ i_{2}^{\ast}\right)  =\mathcal{L}\left(  \left(
i\,_{2}\circ j_{1}\right)  ^{\ast}\right)  $ is an isomorphism. The same
reasoning shows that $\mathcal{L}\left(  i_{1}^{\ast}\circ j_{1}^{\ast
}\right)  $ is an isomorphism as well.

Applying the Leray functor to the diagram (\ref{eq:cont_prop_pf_diag}), we see
that $\mathcal{L}\left(  j^\ast\right)  $ is an isomorphism. Letting $\varphi=j $,
we conclude that the conditions (\ref{eq:def_dominates1}) and
(\ref{eq:def_dominates2}) are satisfied (the second one is satisfied because $j$ is a composition of inclusion maps). Therefore $C_{\Lambda}\left(
K_{\alpha},F_{\alpha}\right)  \simeq$ $C_{\Lambda}\left(  K_{\beta},F_{\beta
}\right)  $ in view of Lemma \ref{lem:dominates_isom}, which completes the proof.
\end{proof}


\section{Proof of the main result \label{sec:proof_cont_thm}}

In this section we will give the proof of Theorem \ref{thm:cont_thm}. We will first give the idea of the proof and
then follow with detailed constructions. 

In the
proof we will consider the sets $D^{u}$ and $\partial D^{u}$
defined as
\begin{align}
D^{u}  &  :=\left\{  \left(  \theta;x,y\right)  \in D:y=0,\left\Vert
x\right\Vert _{u}\leq 1\right\}  ,\label{eq:Du-def}\\
\partial D^{u}  &  :=\left\{  \left(  \theta;x,y\right)  \in D:y=0,\left\Vert
x\right\Vert _{u}=1\right\}  \label{eq:Du-bd-def}.%
\end{align}
(In
case $s=0$ we need to modify the above definition and omit $y$ coordinate, but
the proof will be analogous). Observe that the pair $(D^u, \partial D^u)$ is nothing but the Thom space of the bundle $\pi^u: E^u \rightarrow \Lambda$ (see Section \ref{subsec:cohomology_bundles}). 

We start by showing that if $D\overset{f}{\implies}D$, then $\tilde{f}%
^{\ast}:H^{\ast}\left(  \Lambda\right)  \rightarrow H^{\ast}\left(
\Lambda\right)  $ is an isomorphism. If we consider the map $h_{1}$ after the
homotopy (\ref{eq:cov_rel_homotopy}), then trivially $D\overset{h_{1}%
}{\implies}D$, so in particular we will obtain that $\left(  h_{1}\right)
_{\Lambda}^{\ast}:H^{\ast}\left(  \Lambda\right)  \rightarrow H^{\ast}\left(
\Lambda\right)  $ is also an isomorphism.

In the next part of the section we consider the map $h_{1}$, for which we will
construct an index pair $P=\left(  P_{1},P_{2}\right)  \in \operatorname{IP}\left(
D,h_{1}\right)  $ that deformation retracts onto $\left(  D^{u},\partial
D^{u}\right)$. Utilising properties of Thom classes, we will prove that the index map for
$h_{1}$ for this index pair is an
isomorphism. In view of Remark \ref{rem:con_ind_isom}, this will mean that the
Conley index of $\Lambda$ as an invariant set for $h_{1}:D\rightarrow E$ is
equal to $\left(  H^{\ast}\left(  P_{1},P_{2}\right)  ,I_{P}\right)  $.

These facts will imply that for the map $h_{1}$, the
external product%
\[
\ast:H^{\ast}\left(  \Lambda\right)  \times H^{\ast}\left(  P_{1}%
,P_{2}\right)  \rightarrow H^{\ast}\left(  P_{1},P_{2}\right)  ,
\]
is given by%
\[
\alpha\ast v=\pi_{P_{1}}^{\ast}\left(  \alpha\right)  \smile v.
\]
Using this and the fact that $\left(  P_{1},P_{2}\right)  $ is homotopy
equivalent to $\left(  D^{u}, \partial{D}^{u}\right) ,$ we will show, using the Thom isomorphism theorem (see Theorem
\ref{thm:thom_isom})), that for some $p_{T}\in H^{\ast}\left(
P_{1},P_{2}\right)$, the map
\[
H^{\ast}\left(  \Lambda\right)  \ni u\mapsto u\ast p_{T}\in H^{\ast}\left(
P_{1},P_{2}\right)
\]
is an isomorphism.

In the second part of the construction we will consider the map $f$. The above fact together with the
continuation property of the Conley index will imply that for an arbitrary
index pair $R$ for $\operatorname{Inv}(D,f)$ contained in $D$, there exists an $\left[
r_{T}\right]  \in\operatorname{gim}\left(  I_{R}^{\prime}\right)  $ such that
the map%
\begin{equation}
H^{\ast}\left(  \Lambda\right)  \ni \alpha\mapsto \alpha\ast\left[  r_{T}\right]
\in\operatorname{gim}\left(  I_{R}^{\prime}\right)  \label{eq:*-iso-temp}%
\end{equation}
is an isomorphism. Thus $C_\Lambda(\operatorname{Inv}(D,f), f)$ will turn out to be the free left module over $H^\ast(\Lambda)$ with a basis consisting of one element. 
The injectivity of \eqref{eq:cont_thm_hom} will follow from this algebraic fact.

We now carry out this plan in greater detail. The proof will be preceded by a
few preparatory lemmas. We fix a coefficient ring $R$ as in Remark \ref{rem:cont_thm_coeff} or Remark \ref{rem:cont_thm_coeff_2}.

\begin{lemma}
\label{lem:cov_relations_isom} If $D\overset{f}{\Longrightarrow}D,$ then
$\tilde{f}^{\ast}$ is an isomorphism. In particular, $\operatorname{gim}%
(\tilde{f}^{\ast})\simeq H^{\ast}\left(  \Lambda\right)  .$
\end{lemma}
\begin{proof}
If $f$ is a homeomorphism, $\tilde{f}^*$ is an isomorphism, since $\pi^*\circ\tilde{f}^* = f^* \circ \pi^*$ in view of Proposition \ref{lem:admissibility} and $\pi$ induces an isomorphism on cohomology (indeed, $i_{\Lambda,E}$ is its inverse up to a homotopy). Hence the assertion follows. If $f$ is merely continuous, we reason as follows. Note that if we restrict the homotopy (\ref{eq:cov_rel_homotopy}) from the
covering relations to $\Lambda,$ we can infer that the maps $\tilde{f}$ and
$\eta$ are homotopic. Hence
\[
\tilde{f}^{\ast}=\eta^{\ast},
\]
so it suffices to show that the induced map $\eta^{\ast}$ is an isomorphism. We can use Proposition
\ref{prop:degree} and the degree condition from the covering relations to obtain that
$\eta^{\ast}$ is an isomorphism on cohomology with the choice of coefficients as in Remark \ref{rem:cont_thm_coeff} and Remark \ref{rem:cont_thm_coeff_2}. This
completes the proof.
\end{proof}

Since for the map $h_{1}$ after the homotopy \eqref{eq:cov_rel_homotopy} we have
$D\overset{h_{1}}{\implies}D$ (the homotopy needed for this covering is the
identity), from Lemma \ref{lem:cov_relations_isom} we obtain the following corollary.
\begin{corollary}
 The map $\tilde{h}
_{1}^{\ast}:H^{\ast}\left(  \Lambda\right)  \rightarrow
H^{\ast}\left(  \Lambda\right)$ is an isomorphism.
\end{corollary}

Recall that a pair $(Y,B)\subset\left(  X,A\right)  $ is a \emph{deformation
retract} of $\left(  X,A\right)  $ if there exists a homotopy $H:\left[
0,1\right]  \times X\rightarrow X$ such that%
\begin{align*}
H\left(  0,\cdot\right)   &  =id_{X},\qquad H\left(  1,X\right)  \subset
Y,\qquad H\left(  1,A\right)  \subset B,\\
H\left(  t,\cdot\right)  |_{Y}  &  =id_{Y}\text{\quad and\quad}H\left(
t,A\right)  \subset A\qquad\text{for all }t\in\left[  0,1\right]  .
\end{align*}

\begin{lemma}
\label{lem:cov_relations_ind_pair} If $D\overset{f}{\Longrightarrow}D$, then
there exists an index pair $P=\left(  P_{1},P_{2}\right)  \in\operatorname{IP}%
(D,h_{1})$ such that $\left(  D^{u},\partial
D^{u}\right)  $ is a deformation retract of $P.$
\end{lemma}
\begin{proof}
We will tackle the case of $s\geq1,$ i. e. the stable bundle is nonempty.
The proof in the case $s=0$ is similar.

Observe that $\Lambda=\mathrm{Inv}\left(  D,h_{1}\right)  $ and therefore we
can take the isolating neighbourhood to be $D$. Due to the compactness of
$\Lambda$ and the fact that $h_{1}$ is fibrewise linear, we can choose a
constant $C\in\left(  0,1\right)  $ such that for
\begin{align*}
P_{1}  &  =\left\{  \left(  \theta;x,y\right)  \in D:\left\Vert y\right\Vert
_{s}\leq1/2\right\}  ,\\
P_{2}  &  =\left\{  \left(  \theta;x,y\right)  \in P_{1}:\left\Vert
x\right\Vert _{u}\geq C\right\}  ,
\end{align*}
the pair $\left(  P_{1},P_{2}\right)  $ is an index pair for $h_{1}$ with
respect to $D$.

Consider now $D^{u}$ and $\partial D^{u}$ from (\ref{eq:Du-def}) and
(\ref{eq:Du-bd-def}), respectively, to be defined using the above $C$. We can
take the homotopy $H:\left[  0,1\right]  \times P_{1}\rightarrow P_{1}$ for
the deformation retract as%
\[
H\left(  t,\left(  \theta;x,y\right)  \right)  =\left\{
\begin{array}
[c]{ll}%
\left(  \theta;x,\left(  1-t\right)  y\right)  & \qquad\text{if }\left\Vert
x\right\Vert _{u}\leq C\\
\left(  \theta;\left(  \left(  1-t\right)  +tC\left\Vert x\right\Vert
_{u}^{-1}\right)  x,\left(  1-t\right)  y\right)  & \qquad\text{otherwise.}%
\end{array}
\right.
\]
This concludes our proof.
\end{proof}

\begin{lemma}
\label{lem:cov_relations_con_ind} Assume that $P=\left(  P_{1},P_{2}\right)
\in\operatorname{IP}(D,h_{1})$ is the index pair from Lemma
\ref{lem:cov_relations_ind_pair}. Then there exists a cohomology class
$p_{T}\in H^{\ast}\left(  P_{1},P_{2}\right)  $ such that the map $$F:H^{\ast}\left(
\Lambda\right)  \rightarrow H^{\ast+u}\left(  P_{1},P_{2}\right)$$ defined
by
\begin{equation}
F\left(  \alpha\right)  =\pi_{P_1}^{\ast}(\alpha)\smile p_{T}\label{eq:lem_cov_relations_con_ind_isom}
\end{equation}
is an isomorphism. Moreover, the index map $I_{P}$ for $h_{1}$ is an
isomorphism and therefore $\operatorname{gim}\left(  I_{P}^{\prime}\right)
\simeq H^{\ast}\left(  P_{1},P_{2}\right)  .$
\end{lemma}
\begin{proof}
Let $H:\left[  0,1\right]  \times P_{1}\rightarrow P_{1}$ be the homotopy from
the proof of Lemma \ref{lem:cov_relations_ind_pair}, required for $\left(
D^{u},\partial D^{u}\right)  $ to be the deformation retract of $\left(
P_{1},P_{2}\right)  $. Let 
$r  :=H\left(  1,\cdot\right)$ and observe that
\begin{equation}
\pi_{D^{u}}\circ r=\pi_{P_{1}} \label{eq:lem_cov_rel_ind_pair_1}.%
\end{equation}

Take an arbitrary $p\in H^{\ast}\left(  P_{1},P_{2}\right)  .$ Since $r$ is a
deformation retraction, $r^{\ast}$ is an isomorphism. Therefore there exists a unique
$q\in H^{\ast}\left(  D^{u},\partial D^{u}\right)  $ satisfying
\[
p=r^{\ast}\left(  q\right)
\]
By the Thom isomorphism theorem (see Theorem \ref{thm:thom_isom}), we can
write $q$ as%
\[
q=\pi_{D^{u}}^{\ast}\left(  \alpha\right)  \smile v_{T},
\]
where $v_{T}\in H^{\ast}\left(  D^{u},\partial D^{u}\right)  $ is the Thom
class and $\alpha\in H^{\ast}\left(  \Lambda\right)  $ is uniquely determined by $q$. Therefore
\begin{align*}
p  &  =r^{\ast}\left(  q\right) =r^{\ast}\left(  \pi_{D^{u}}^{\ast}\left(  \alpha\right)  \smile v_{T}\right)
=\left(  \pi_{D^{u}}\circ r\right)  ^{\ast}\left(  \alpha\right)  \smile
r^{\ast}\left(  v_{T}\right) =\pi_{P_{1}}^{\ast}\left(  \alpha\right)  \smile r^{\ast}\left(  v_{T}\right),
\end{align*}
in view of \eqref{eq:lem_cov_rel_ind_pair_1}.%

Letting
\begin{equation}
p_{T}=r^{\ast}\left(  v_{T}\right)  , \label{eq:lem_cov_rel_ind_pair_6}%
\end{equation}
we finally see that each $p\in H^{\ast}\left(  P_{1},P_{2}\right)  $ can be
uniquely represented as%
\begin{equation}
p=\pi_{P_{1}}^{\ast}\left(  u\right)  \smile p_{T},
\label{eq:lem_cov_rel_ind_pair_2}%
\end{equation}
which proves that \eqref{eq:lem_cov_relations_con_ind_isom} is an isomorphism.

We will now prove that $I_{P}$ for $h_{1}$ is an isomorphism. We will first
consider the case when $s=0$, i. e. there is no stable direction and the map
$h_{1}:E\rightarrow E$ is given by%
\[
h_{1}\left(  \theta;x\right)  =\left(  \eta\left(  \theta\right)  ;A_{\theta
}x\right)  .
\]

Denote $$P_{i,\theta} = P_i \cap E_\theta,$$ $$T(P_i)_\theta = T(P_i) \cap E_\theta$$ for $i = 1,2$, $P_\theta = (P_{1,\theta},P_{2,\theta})$ and $T(P)_\theta = (T(P_1)_\theta,T(P_2)_\theta).$ Note that it follows from the formula \eqref{eq:lem_cov_rel_ind_pair_2} that every $p \in H^u(P_1,P_2)$ is some multiple of $p_T$. Thus
$I_P(p_T) = \beta p_T,$
where $\beta \in R$, the coefficient ring. We will first prove that $\beta = \pm1$. 

Since the Thom class $v_T$ of $E^u$ restricts to the generator $$v_T\vert_\theta = i_{(D^u_\theta,\partial D^u_\theta), (D^u, \partial D^u)}^\ast \circ v_T$$ of $H^u(D^u_\theta, \partial D^u_\theta)$ for every $\theta \in \Lambda$ and the retraction $r$ preserves fibers, so does $p_T$ restrict to the generator of $H^u(P_{1,\theta},P_{2,\theta})$. We now use the assumption that $s = 0$. Observe that in this case the pair $(T(P_1), T(P_2))$ deformation retracts onto $(D^u, \partial D^u)$, which can be proven in a similar way as Lemma \ref{lem:cov_relations_ind_pair}. Thus we also have an isomorphism $H^0(\Lambda) \simeq H^u(T(P_1), T(P_2))$ and the generator $p'_T \in H^u(T(P_1), T(P_2))$ restricting to the generator of $H^u(T(P_1)_{\theta},T(P_2)_{\theta})$ for every $\theta \in \Lambda$.

It follows that we have
$$(h_1)^\ast_{P,T(P)}(p_T') = \beta' p_T,$$
for some $\beta' \in R.$
Then, using the fact that $h_1$ preserves fibers, we obtain 
\begin{align*}
\left(h_1\right)^*_{P_\theta,T(P)_\theta}\left(p'_T\vert_{\eta(\theta)}\right) &= \left(\left(h_1\right)^*_{P_\theta,T(P)_\theta}  \circ i_{P_{\eta(\theta)}, T(P)_{\eta(\theta)}}^\ast\right)(p'_T) \\ &= \left(i_{P_{\eta(\theta)}, T(P)_{\eta(\theta)}} \circ \left(h_1\right)_{P_\theta,T(P)_\theta}\right)^\ast(p'_T) \\
&= \left(\left(h_1\right)_{P,T(P)} \circ i_{P_{\theta}, T(P)_{\theta}}\right)^\ast(p_T') \\ &= i_{{P_\theta},T(P)_\theta}^\ast\left(\left(h_1\right)^\ast_{P,T(P)}(p'_T)\right) = \beta'p_T\vert_{\theta}.
\end{align*}

Since $P_{1,\theta}/P_{2,\theta}$ and $T\left(P_{1}\right)_{\eta(\theta)}/T\left(
P_{2}\right)_{\eta(\theta)}$ are both homeomorphic to the unit sphere, from Proposition \ref{thm:lin_isom_degree} we deduce that$$\left(h_1\right)^*_{P_\theta,T(P)_\theta}\left(p'_T\vert_{\eta(\theta)}\right) = \left(\det A_{\theta}\right)p_T\vert_{\theta} =\pm p_T\vert_{\theta}.$$ Hence $\beta' = \pm1$ and consequently $\beta = \pm1.$

Using the formula \eqref{eq:pf_ext_prod_well_defined}, for $\alpha\in H^{\ast}(\Lambda)$ we have
\begin{align*}
I_{P}\left(\pi_{P_1}^{\ast}(\alpha)  \smile p_{T}\right) = \pi^*_{P_1}(\eta^*(\alpha)) \smile I_P(p_T) = r\left(\pi^*_{P_1}(\eta^*(\alpha)) \smile p_T \right)
\end{align*}
and therefore%
\[
I_{P}=\beta F\circ\eta^{\ast}\circ F^{-1}.
\]
Since $\beta $ is invertible by our choice of coefficients, it follows that $I_{P}$ is an isomorphism.

We will now tackle the case when $s>0.$ Let $$P'=\left(  P'%
_{1},P_2'\right)  =\left(  \pi^{u}\left(  P_{1}\right)  ,\pi
^{u}\left(  P_{2}\right)  \right),$$ where $\pi^{u}:E\rightarrow E^{u}$ is
the projection onto the bundle $E^{u}.$ Clearly the pair $\left( 
{P}'_{1},{P}'_{2}\right)  $ is a deformation retract of $\left(
P_{1},P_{2}\right),$ so%
\[
\pi^{u\ast}:H^{\ast}\left(  \tilde{P}_{1},\tilde{P}_{2}\right)  \rightarrow
H^{\ast}\left(  P_{1},P_{2}\right)
\]
is an isomorphism. If we let ${h}'_{1}:E^{u}\rightarrow E^{u}$ be the
restriction of $h_{1}$ to $E^{u},$we conclude that $\left(  P'
_{1},{P}'_{2}\right)  $ is an index pair for ${h}'_{1}$ with respect
to $D^{u}.$ Moreover, by the previous case the index map for this index pair
\[
I_{{P}'}:H^{\ast}\left(  {P}'_{1},{P}'_{2}\right)  \rightarrow
H^{\ast}\left(  {P}'_{1},{P}'_{2}\right)
\]
is an isomorphism.

Observe that the index map satisfies $$I_P = f^*_{P,S(P)}\circ i_{P,S(P)}^{\ast-1},$$ where $S(P) = \left(S(P_1), S(P_2)\right)$ was defined in Section \ref{subsubsec:con_ind_def}. Applying the cohomology functor to the commutative diagram%
\[%
\begin{array}
[c]{lllll}%
\left(  P_1,P_2\right)  & \overset{\left(  {h}%
_{1}\right)_{P,S\left(  P\right)}}{\longrightarrow} & \left(  S(P%
_{1}),S(P_2)\right)  & \overset{i_{P,S(P)}%
}{\longleftarrow} & \left( P_1,P_2\right) \\
\quad\downarrow\pi^{u} &  & \quad\quad\downarrow \pi^u&  & \quad\downarrow\pi^{u}\\
\left(  P'_{1},P'_{2}\right)  & \overset{\left(  h'_{1}\right)  _{P',S\left(
P'\right)  }}{\rightarrow} & \left(  S(P'_{1}),S(P'_{2})\right)  & \overset
{i_{P',S(P')}}{\leftarrow} & \left(  P'_{1},P'_{2}\right)
\end{array}
\]
we conclude that the diagram%
\[%
\begin{array}
[c]{lll}%
H^{\ast}\left(P\right)  & \overset{{I}_{P}%
}{\leftarrow} & H^{\ast}\left( P\right) \\
\quad \uparrow\pi^{u\ast} &  & \quad\uparrow\pi^{u\ast}\\
H^{\ast}\left( P'\right)  & \overset{I_{P'}}{\leftarrow} & H^{\ast
}\left(  P'\right)
\end{array}
\]
commutes. This implies that $I_{P}$ is an isomorphism, which completes the proof.
\end{proof}

According to Lemma \ref{lem:cov_relations_con_ind}, the module $C_{\Lambda
}\left(  \Lambda,h_{1}\right)  $ has a particularly simple structure -- it has
a basis consisting of a single element $p_{T}.$

\begin{lemma}
\label{lem:cov_relations_ind_pair_isom} If $D\overset{f}{\Longrightarrow}D,$ then for any isolating neigbhourhood $N \subset D$ for $\operatorname{Inv(D,f)}$ and any 
index pair $R\in\operatorname{IP}\left(  N,f\right)$ there exists $\left[
r_{T}\right]  \in$ $\operatorname{gim}\left(  I_{R}^{\prime}\right)  $ such
that the map%
\[
H^{\ast}\left(  \Lambda\right)  \ni \alpha\mapsto \alpha\ast\left[  r_{T}\right]
\in\operatorname{gim}\left(  I_{R}^{\prime}\right)
\]
is an isomorphism.
\end{lemma}
\begin{proof}
Choose an isolating neighborhood $N$ and an index pair $P\in\operatorname{IP}%
(N,h_{1})$ as in Lemma \ref{lem:cov_relations_con_ind}. Since $H^{\ast}\left(
P_{1},P_{2}\right)  $ is a free left module over $H^{\ast}\left(
\Lambda\right)  $ with a basis consisting of a single element $p_{T}$, by
Theorem \ref{thm:cont_prop} we conclude that the same is true for
$\operatorname{gim}\left(  I_{R}^{\prime}\right).$ Thus there exists some
element $\left[  r_{T}\right]  \in$ $\operatorname{gim}\left(  I_{R}^{\prime
}\right)$ such that each $\left[  r\right]  \in\operatorname{gim}\left(
I_{R}^{\prime}\right)  $ can be written uniquely as%
\[
\left[  r\right]  =\alpha\ast\left[  r_{T}\right]  ,
\]
for some $\alpha\in H^{\ast}\left(  \Lambda\right)  .$ This completes the proof.
\end{proof} \\
\begin{proof}
[Proof of Theorem \ref{thm:cont_thm}]To show that the homomorphism
(\ref{eq:cont_thm_hom}) is injective, suppose that
\begin{equation}
\tilde{\pi}^\ast(\alpha) = \pi_{\operatorname{Inv}(D,f)}^{\ast}(\alpha)=0\label{eq:cont_thm_eq1}%
\end{equation}
for some $\alpha\in H^{\ast}(\Lambda).$ Let $Q=(Q_1,Q_2) \in \operatorname{IP}(D,f)$ be any index pair and observe that 
$$i^\ast_{\operatorname{Inv}(D,f),Q_1}\left(\pi_{_{Q_1}}^\ast(\alpha)\right) = \tilde{\pi}^\ast(\alpha) = 0.$$
By Proposition \ref{prop: lim_ind_pairs} and Theorem \ref{thm:tautness}, we deduce that there exists an isolating neighborhood $D' \subset D$ and an index pair $R = (R_1, R_2) \in \operatorname{IP}(D',f)$ for $\operatorname{Inv}(D,f)$ such that $R \subset Q$ and
$$\pi_{R_1}^\ast (\alpha)= i^\ast_{R_1,Q_1}(\pi_{_{Q_1}}^\ast(\alpha)) = 0.$$

Take $\left[  w_{T}\right]  \in\operatorname{gim}\left(  I_{R}^{\prime
}\right)  $ from Lemma \ref{lem:cov_relations_ind_pair_isom}. Then the bilinearity of the
relative cup product gives%
\[
\pi_{R_{1}}^{\ast}(\alpha)\smile w_{T}=0,
\]
hence%
\[
\alpha\ast\left[  w_{T}\right]  =\left[  \pi_{{R_{1}}}^{\ast}(\alpha)\smile w_{T}\right]
=0.
\]
Because the homomorphism $v\mapsto v\ast\left[  w_{T}\right]  $ is an isomorphism in view
of Lemma \ref{lem:cov_relations_ind_pair_isom}, it follows that $\alpha = 0$. This
completes the proof.
\end{proof}

\section{Conclusions}\label{sec:conclusions}
In this paper, we have generalized Floer's topological persistence theorem to a non-perturbative setting, using topological conditions formulated in terms of covering relations. Our approach does not require the prior existence of a normally hyperbolic invariant manifold, nor the smoothness or invertibility of the underlying map. Using the Conley index we establish an invariant set, whose cohomology ring contains a subring isomorphic to that of a reference manifold

We expect that similar techniques can be applied to study the persistence of stable and unstable sets associated with normally hyperbolic invariant manifolds under non-smooth perturbations. Moreover, the methods developed here may provide tools for investigating the topology of intersections between such sets.  We plan to pursue this in our future work.

\section*{Acknowledgements}
We would like to thank Rafael de la Llave for suggesting this problem to us.

The research of W.S. was partially supported by the program ``Excellence Initiative -- Research University'' (IDUB) at AGH and a subsidy from the Polish Ministry of Science and Higher Education. The research of M.J.C. was partially supported by the National Science Center Poland (NCN) grant 2021/41/B/ST1/00407 and a subsidy from the Polish Ministry of Science and Higher Education.

\bibliographystyle{plain}
\bibliography{paper}

\end{document}